\documentclass[a4paper]{amsart}

\usepackage{amsfonts}
\usepackage{amssymb}
\usepackage{amsmath}
\usepackage{hyperref}
\usepackage{mathrsfs}
\usepackage{centernot}
\usepackage[all]{xy}

\newtheorem{thm}{Theorem}[section]
\newtheorem{lem}[thm]{Lemma}
\newtheorem{prp}[thm]{Proposition}
\newtheorem{cor}[thm]{Corollary}
\newtheorem{dfn}[thm]{Definition}

\newtheorem{cnj}[thm]{Conjecture}

\newtheorem{que}{Question}

\newtheorem{baseexample}[thm]{Example} 
\newtheorem{baseremark}[thm]{Remark} 

\newenvironment{example}
{\begin{baseexample}\rm}{\end{baseexample}}
\newenvironment{remark}
{\begin{baseremark}\rm}{\end{baseremark}}

\newcommand{\term}[1]{\emph{#1}}
\newcommand{\rem}[1]{}

\newcommand{\N}{\mathbb{N}}
\newcommand{\Q}{\mathbb{Q}}
\newcommand{\R}{\mathbb{R}}
\newcommand{\Z}{\mathbb{Z}}

\newcommand{\frakSmall}{
\newcommand{\fraka}{{\mathfrak{a}}}
\newcommand{\frakb}{{\mathfrak{b}}}
\newcommand{\frakc}{{\mathfrak{c}}}
\newcommand{\frakd}{{\mathfrak{d}}}
\newcommand{\frake}{{\mathfrak{e}}}
\newcommand{\frakf}{{\mathfrak{f}}}
\newcommand{\frakg}{{\mathfrak{g}}}
\newcommand{\frakh}{{\mathfrak{h}}}
\newcommand{\fraki}{{\mathfrak{i}}}
\newcommand{\frakj}{{\mathfrak{j}}}
\newcommand{\frakk}{{\mathfrak{k}}}
\newcommand{\frakl}{{\mathfrak{l}}}
\newcommand{\frakm}{{\mathfrak{m}}}
\newcommand{\frakn}{{\mathfrak{n}}}
\newcommand{\frako}{{\mathfrak{o}}}
\newcommand{\frakp}{{\mathfrak{p}}}
\newcommand{\frakq}{{\mathfrak{q}}}
\newcommand{\frakr}{{\mathfrak{r}}}
\newcommand{\fraks}{{\mathfrak{s}}}
\newcommand{\frakt}{{\mathfrak{t}}}
\newcommand{\fraku}{{\mathfrak{u}}}
\newcommand{\frakv}{{\mathfrak{v}}}
\newcommand{\frakw}{{\mathfrak{w}}}
\newcommand{\frakx}{{\mathfrak{x}}}
\newcommand{\fraky}{{\mathfrak{y}}}
\newcommand{\frakz}{{\mathfrak{z}}}
}

\newcommand{\calCapital}{
\newcommand{\calA}{{\mathcal{A}}}
\newcommand{\calB}{{\mathcal{B}}}
\newcommand{\calC}{{\mathcal{C}}}
\newcommand{\calD}{{\mathcal{D}}}
\newcommand{\calE}{{\mathcal{E}}}
\newcommand{\calF}{{\mathcal{F}}}
\newcommand{\calG}{{\mathcal{G}}}
\newcommand{\calH}{{\mathcal{H}}}
\newcommand{\calI}{{\mathcal{I}}}
\newcommand{\calJ}{{\mathcal{J}}}
\newcommand{\calK}{{\mathcal{K}}}
\newcommand{\calL}{{\mathcal{L}}}
\newcommand{\calM}{{\mathcal{M}}}
\newcommand{\calN}{{\mathcal{N}}}
\newcommand{\calO}{{\mathcal{O}}}
\newcommand{\calP}{{\mathcal{P}}}
\newcommand{\calQ}{{\mathcal{Q}}}
\newcommand{\calR}{{\mathcal{R}}}
\newcommand{\calS}{{\mathcal{S}}}
\newcommand{\calT}{{\mathcal{T}}}
\newcommand{\calU}{{\mathcal{U}}}
\newcommand{\calV}{{\mathcal{V}}}
\newcommand{\calW}{{\mathcal{W}}}
\newcommand{\calX}{{\mathcal{X}}}
\newcommand{\calY}{{\mathcal{Y}}}
\newcommand{\calZ}{{\mathcal{Z}}}
}

\newcommand{\CFont}{\mathscr} 

\DeclareMathOperator{\ann}{ann}
\newcommand{\annl}{\ann^\ell}
\newcommand{\annr}{\ann^r}
 \DeclareMathOperator{\Aut}{Aut}
 \DeclareMathOperator{\Cent}{Cent}

 \DeclareMathOperator{\End}{End}
 
\DeclareMathOperator{\Hom}{Hom} %
\DeclareMathOperator{\id}{id} %
\DeclareMathOperator{\im}{im} %
\DeclareMathOperator{\Jac}{Jac} %

\newcommand{\op}{\mathrm{op}}

\DeclareMathOperator{\soc}{soc} %
\DeclareMathOperator{\Spec}{Spec} %
\newcommand{\what}[1]{\widehat{#1}}

\newcommand{\veps}{\varepsilon}
\newcommand{\vphi}{\varphi}

\newcommand{\idealof}{\unlhd} 

\newcommand{\derives}{\Longrightarrow}

\newcommand{\suchthat}{\,:\,}
\newcommand{\where}{\,|\,}

\newcommand{\quo}[1]{\overline{#1}}

\newcommand{\smallSMatII}[4]{\left[\begin{smallmatrix} {#1} & {#2} \\ {#3} &
{#4} \end{smallmatrix}\right]}

\newcommand{\Trings}[1]{\left< #1 \right>}

\newcommand{\nGal}[2]{\mathrm{Gal}({#1}/{#2})}

\newcommand{\nMat}[2]{\mathrm{M}_{#2}(#1)}

\newcommand{\invlim}{\underleftarrow{\lim}\,}

\newcommand{\dpar}[1]{(\!(#1)\!)}
\newcommand{\dbrac}[1]{[[#1]]}

\newcommand{\units}[1]{{#1^\times}}
\newcommand{\ideal}[1]{\left<#1\right>}

\newcommand{\ids}[1]{\mathrm{E}(#1)}

\newcommand{\cHom}{\Hom_{\mathrm{c}}}
\newcommand{\cEnd}{\End_{\mathrm{c}}}
\newcommand{\cAut}{\Aut_{\mathrm{c}}}

\newcommand{\LTR}{{\CFont{L}\!\!\CFont{T}_2}}

\DeclareMathOperator{\Ball}{B}

\calCapital %
\frakSmall

\begin{document}


\title{Semi-Invariant Subrings}
\author{Uriya A.\ First}
\address{Department of Mathematics, Bar-Ilan University, Ramat-Gan}
\email{uriya.first@gmail.com}
\thanks{This research was partially supported by an Israel-US BSF grant \#2010/149.}
\date{\today}
\maketitle

\begin{abstract}
    We say that a subring $R_0$ of a ring $R$ is \emph{semi-invariant}
    if $R_0$ is the ring of invariants in $R$ under some set of
    ring endomorphisms of some ring containing $R$.
    We show that $R_0$ is semi-invariant
    if and only if there is a ring $S\supseteq R$ and a set $X\subseteq S$ such that $R_0=\Cent_R(X):=\{r\in R
    \suchthat xr=rx~ \forall x\in X\}$; in particular, centralizers of subsets of $R$ are semi-invariant subrings.

    We prove that a semi-invariant subring $R_0$ of a semiprimary (resp.\ right perfect) ring $R$
    is again semiprimary (resp.\ right perfect) and satisfies $\Jac(R_0)^n\subseteq\Jac(R)$ for
    some $n\in\N$. This result holds for other families
    of semiperfect rings, but the semiperfect analogue fails in general. In order to overcome
    this, we specialize to Hausdorff \emph{linearly topologized} rings and consider
    \emph{topologically semi-invariant} subrings. This enables us to
    show that any topologically semi-invariant subring (e.g.\
    a centralizer of a subset) of a semiperfect ring that can be endowed with a ``good'' topology (e.g.\
    an inverse limit of semiprimary rings) is semiperfect.

    Among the applications: (1) The center of a semiprimary (resp.\ right perfect) ring
    is semiprimary (resp.\ right perfect). (2) If $M$ is a finitely presented module over a ``good''
    semiperfect ring (e.g.\ an inverse limit of semiprimary rings),
    then $\End(M)$ is semiperfect, hence $M$ has a Krull-Schmidt decomposition. (This generalizes results of Bjork
    and Rowen; see \cite{Bj71B}, \cite{Ro87}, \cite{Ro86}.) (3) If $\rho$ is a representation of a monoid or
    a ring over a module with a ``good'' semiperfect endomorphism ring (in the sense of (2)), then $\rho$ has a Krull-Schmidt decomposition.
    (4) If $S$ is a ``good'' commutative semiperfect ring and $R$ is an $S$-algebra that is f.p.\ as an $S$-module,
    then $R$ is semiperfect.
    (5) Let $R\subseteq S$ be rings and let $M$ be a right $S$-module. If $\End(M_R)$ is semiprimary (resp.\ right perfect), then
    $\End(M_S)$ is semiprimary (resp.\ right perfect).

\end{abstract}

\section{Preface}
\label{section:preface}

    Throughout, all rings are assumed to have a unity and ring
    homomorphisms are required to preserve it.
    Subrings are assumed to have the same unity as the ring containing them.
    Given a ring $R$, denote its set of invertible elements by $\units{R}$, its Jacobson radical by
    $\Jac(R)$, its set of idempotents by $\ids{R}$ and its center by $\Cent(R)$. We let $\End(R)$ (resp.\ $\Aut(R)$)
    denote the set of ring homomorphisms (resp.\ isomorphisms) from $R$ to itself. If $X\subseteq R$
    is any set, then its right (left) annihilator in $R$
    will be denoted by $\annr_R X$ ($\annl_R X$). The subscript $R$ will
    be dropped when understood from the context.
    A \emph{semisimple} ring always means a semisimple artinian ring. For a prime number $p$, we
    let $\Z_p$, $\Q_p$ and $\Z_{\ideal{p}}$ denote the $p$-adic integers, $p$-adic numbers and $\Z$
    localized (but not completed) at $p\Z$, respectively.

    \medskip

    Let $R$ be a ring and let $J=\Jac(R)$. Recall that $R$ is \emph{semilocal} if $R/J$ is semisimple. If in addition
    $J$ is idempotent lifting, then $R$ is called \emph{semiperfect}. For a detailed discussion about semiperfect rings,
    see \cite[\S2.7]{Ro88} and \cite{Ba60}. Semiperfect rings play an important role in representation theory and module
    theory because of the Krull-Schmidt Theorem.
    Recall that an object $A$ in an additive category $\CFont{A}$ is said to have a \emph{Krull-Schmidt decomposition}
    if it is a sum of (non-zero) indecomposable objects and any two such decompositions are the same up to isomorphism and reordering.

    \begin{thm}{\bf (Krull-Schmidt, for Categories).}
        Let $\CFont{A}$ be an additive category in which all idempotents split (e.g.\ an abelian category) and let $A\in \CFont{A}$.
        If $\End_{\CFont{A}}(A)$ is semiperfect, then $A$ has a Krull-Schmidt decomposition and the endomorphism ring
        of any indecomposable summand of $A$ is local.
    \end{thm}

    Generalizations of this theorem and counterexamples of some natural variations are widely studied (e.g.\ see \cite{FHL95},\cite{BFR01},\cite{ABF09},\cite{FaEcKo10} and also
    \cite{Fa03})
    and there has been considerable interest in finding rings over which all finitely presented modules
    have a Krull-Schmidt decomposition (e.g.\ \cite[\S6]{Sw60},\cite{Bj71B},\cite{Ro86},\cite{Ro87},\cite{Vamos90}; theorem
    \ref{RING:TH:application-to-fp-modules-top-case}(iii) below generalizes all these references except the last).

    \begin{example} \label{RING:EX:semiperfect-rings}
        Semiperfect rings naturally appear upon taking completions:
        \begin{enumerate}
            \item[(1)] Let $R$ be a semilocal ring
            and let $J=\Jac(R)$. Then the $J$-adic completion of $R$, $\invlim\{R/J^n\}_{n\in\N}$, is well known to be semiperfect.
            If the natural map $R\to \invlim\{R/J^n\}_{n\in\N}$ is an isomorphism, then $R$ is called \emph{complete semilocal}.
            Such rings (especially noetherian or with Jacobson radical f.g.\ as a right ideal) appear in various areas
            (e.g.\ \cite{Ma58}, \cite{Hi60}, \cite[\S6]{Sw60}, \cite{Ro87}).
            \item[(2)] Let $R$ be a commutative noetherian domain,
            let $A$ be an $R$-algebra that is finitely generated as an $R$-module and let $P\in\Spec(R)$.
            Then the completion of $A$ at $P$ is semiperfect (and noetherian). (See \cite[\S6]{MaximalOrders};
            This assertion can also be shown using the results of this paper.)
        \end{enumerate}
    \end{example}

    Let $R$ be any ring and let $R_0\subseteq R$ be a subring.
    \begin{enumerate}
        \item[(a)] Call $R_0$ a \emph{semi-invariant} subring if there is a ring $S\supseteq R$
            and a set $\Sigma\subseteq\End(S)$
            such that $R_0=R^\Sigma:=\{r\in R\suchthat \sigma(r)=r~\forall \sigma\in\Sigma\}$
            (elements of $\Sigma$ are not required to be injective nor surjective).
            The \emph{invariant} subrings of $R$ are the subrings for which we can choose $S=R$.
        \item[(b)] Call $R_0$ a \emph{semi-centralizer} subring if there is a ring $S\supseteq R$ and a set $X\subseteq S$
            such that $R_0=\Cent_R(X):=\{r\in R\suchthat rx=xr~\forall x\in X\}$.
            If we can choose $S=R$, then $R_0$ is a \emph{centralizer} subring.
        \item[(c)] Recall that $R_0$ is \emph{rationally closed} in $R$ if $\units{R}\cap R_0=\units{R_0}$. That is,
            elements of $R_0$ that are invertible in $R$ are also invertible in $R_0$.
    \end{enumerate}
    Semi-centralizer and semi-invariant subrings are clearly rationally closed. The latter were studied (for semilocal $R$)
    in \cite{CaDi93} and invariant subrings (w.r.t.\ an arbitrary set) were considered in \cite{Bj71B}. However,
    the notion of semi-invariant subrings appears to be new.

    The purpose of this paper is to study semi-invariant subrings of semiperfect rings where our motivation
    comes from the Krull-Schmidt theorem
    and the following observations, verified in sections \ref{section:semi-inv}:
    \begin{enumerate}
        \item[(1)] For any ring $R$, a subring of $R$ is semi-invariant if and only if it is semi-centralizer. In particular,
        all centralizers of subsets of $R$ are semi-invariant subrings.

        \item[(2)] If $R\subseteq S$ are rings and $M$ is a right $S$-module, then $\End(M_S)$ is a semi-invariant subring of $\End(M_R)$.

        \item[(3)] If $M$ is a finitely presented right $R$-module, then $\End(M_R)$ is a quotient of a semi-invariant subring of $\nMat{R}{n}\times
        \nMat{R}{m}$
        for some $n,m$.
    \end{enumerate}
    While in general semi-invariant subrings of semiperfect rings need not be semiperfect (see Examples
    \ref{RING:EX:inv-subring-semiperfect-case-I}-\ref{RING:EX:inv-subring-semiperfect-case-III} below),
    we show that this is true for special families of semiperfect rings, e.g.\ for semiprimary and right perfect rings
    (Theorem \ref{RING:TH:main-res-general}; see Section \ref{section:facts} for definitions).
    In addition, if the ring in question is \emph{pro-semiprimary}, i.e.\ an inverse limit of semiprimary rings
    (e.g.\ the rings of Example \ref{RING:EX:semiperfect-rings}),
    then its \emph{topologically semi-invariant} subrings (e.g.\ centralizer subrings; see Section \ref{section:top-rings} for definition)
    are semiperfect. This actually holds under milder assumptions regarding whether
    the ring can be endowed with a ``good'' topology; see Theorems \ref{RING:TH:main-res-for-quasi-pi-reg-rings}
    and \ref{RING:TH:pro-semiprim-full-result}.

    Our results together with the previous observations and the Krull-Schmidt Theorem lead to numerous applications including:
    \begin{enumerate}
        \item[(1)] The center and any maximal commutative subring
        of a semiprimary (resp.\ right perfect, semiperfect and pro-semiprimary) ring is semiprimary
        (resp.\ right perfect, etc.).
        \item[(2)] If $R$ is a semiperfect pro-semiprimary ring, then all f.p.\ modules over $R$
        have a semiperfect endomorphism ring and hence admit a
        Krull-Schmidt decomposition. If moreover $R$ is right noetherian, then the endomorphism ring of
        a f.g.\ right $R$-modules is pro-semiprimary.
        (This generalizes Swan (\cite[\S6]{Sw60}), Bjork (\cite{Bj71B}) and Rowen (\cite{Ro86}, \cite{Ro87}) and
        also relates to works of V{\'a}mos (\cite{Vamos90}), Facchini and Herbera (\cite{FacHer06}); see Remark \ref{RING:RM:results-related-to-KSD} for more details.)
        \item[(3)] If $S$ is a commutative semiperfect pro-semiprimary ring and $R$ is an $S$-algebra
        that is \emph{Hausdorff} (see Section \ref{section:modules}) and f.p.\ as an $S$-module, then $R$ is semiperfect.
        If moreover $S$ is noetherian, then the Hausdorff assumption is
        superfluous and $R$ is pro-semiprimary, hence the assertions of (2) apply. (The first statement
        is known to hold under mild assumptions for \emph{Henselian} rings; see \cite[Lm.\ 12]{Vamos90}.)
        \item[(4)] If $\rho$ is a representation of a ring or a monoid over a module with a semiperfect pro-semiprimary
        endomorphism ring, then $\rho$ has a Krull-Schmidt decomposition.
        \item[(5)] Let $R\subseteq S$ be rings and let $M$ be a right $S$-module. If $\End(M_R)$ is semiprimary (resp.\ right perfect), then
        so is $\End(M_S)$. In particular, $M$ has a Krull-Schmidt decomposition over $S$. (Compare with \cite[Pr.\ 2.7]{FacHer06}.)
    \end{enumerate}
    Additional applications concern bilinear forms and getting a ``Jordan Decomposition'' for endomorphisms of modules
    with semiperfect pro-semiprimary endomorphism ring. We also conjecture that (3) holds for
    non-commutative $S$ under mild assumptions (see Section
    \ref{section:further-remarks}).

    Other interesting byproducts of our work are the fact that a pro-semiprimary ring
    is an inverse limit of some of its semiprimary quotients and Theorem \ref{RING:TH:closed-submodule-thm} below. (The former assertion
    fails once replacing semiprimary with right artinian; see Example \ref{RING:EX:JacRsquared-can-be-non-open} and
    the comment before it.)

    \begin{remark}
        It is still open whether all semiperfect pro-semiprimary rings are complete semilocal.
        However, this is true for noetherian rings; see  Section \ref{section:rings-with-Hausodroff-mods}.
    \end{remark}

    Section \ref{section:facts} recites some definitions and well-known facts.
    Section \ref{section:semi-inv} presents the basics of semi-invariant subrings; we present five equivalent
    characterizations of them and show that they naturally appear in various situations.
    As all our characterizations use the existence of \emph{some} ambient ring, we ask whether there is a definition avoiding this.
    In Section \ref{section:main-thm},
    we prove that various ring properties pass to semi-invariant subrings, e.g.\
    being semiprimary and being right perfect. Section
    \ref{section:top-rings} develops the theory of T-semi-invariant subrings. The discussion
    leads to a proof that several properties, such as being pro-semiprimary and semiperfect,
    are inherited by T-semi-invariant subrings.
    Section \ref{section:examples} presents counterexamples. We show that semi-invariant subrings of
    semiperfect rings need not be semiperfect, even when the ambient ring is pro-semiprimary.
    In addition, we show that in general none of the properties discussed in sections \ref{section:main-thm} and \ref{section:top-rings}
    pass to rationally closed subrings. The latter implies that there are
    non-semi-invariant rationally closed subrings. In sections \ref{section:applications} and
    \ref{section:modules} we present applications
    of our results (most applications were briefly described above) and in Section
    \ref{section:rings-with-Hausodroff-mods} we specialize them to \emph{strictly pro-right-artinian} rings
    (e.g.\ noetherian pro-semiprimary rings), which are better behaved. Section \ref{section:further-remarks}
    describes some issues that are still open. The appendix
    is concerned with providing conditions implying that the
    topologies $\{\tau_n^M\}_{n=1}^\infty$ defined at Section \ref{section:modules}
    coincide.

\section{Preliminaries}
\label{section:facts}

    This section recalls some definitions and known facts that will be used throughout the paper.
    Some of the less known facts include proofs for sake of completion. If no reference is specified,
    proofs can be found at \cite{Ro88}, \cite{Ba60} or \cite{La99}.
    \smallskip

    Let $R$ be a semilocal ring. The ring $R$ is called semiprimary (right perfect)
    if $\Jac(R)$ is nilpotent (right T-nilpotent\footnote{
            An ideal $I\idealof R$ is right T-nilpotent if for any sequence $x_1,x_2,x_3,\dots\in I$,
            the sequence $x_1,x_2x_1,x_3x_2x_1,\dots$ eventually vanishes.
    }). Since any nil ideal is idempotent lifting,
    right perfect rings are clearly semiperfect.

    \begin{prp}{\bf(Bass' Theorem P, partial).}
        \label{RING:PR:Bass-Thm-P}
        Let $R$ be a ring. Then
        $R$ is right perfect $\iff$ every right
        $R$-module has a projective cover $\iff$ $R$ has DCC on principal \emph{left} ideals.
    \end{prp}

    \begin{prp} \label{RING:PR:semiperfect-equiv-conds}
        Let $R$ be a ring. Then $R$ is semiperfect $\iff$ every right (left)
        f.g.\! $R$-module has a projective cover $\iff$ there are orthogonal
        idempotents $e_1,\dots,e_r\in R$ such that $\sum_{i=1}^re_i=1$ and $e_iRe_i$
        is local for all $i$.
    \end{prp}

    For the next lemma, and also for later usage, let us set some conventions about inverse limits of rings:
    By saying that $\{R_i,f_{ij}\}$ is an \emph{inverse system of rings} indexed by $I$,
    we mean that: (1) $I$ is a \emph{directed set} (i.e.\ a partially ordered set such that for all
    $i,j\in I$ there is $k\in I$ with $i,j\leq k$), (2) $R_i$ ($i\in I$) are rings and $f_{ij}:R_j\to R_i$ ($i\leq j$)
    are ring homomorphisms and (3)
    $f_{ii}=\id_{R_i}$ and $f_{ij}f_{jk}=f_{ik}$ for all $i\leq j\leq k$ in $I$.
    When the maps $\{f_{ij}\}$ are obvious or of little interest, we will drop them from the notation,
    writing only $\{R_i\}_{i\in I}$.
    The \emph{inverse limit} of $\{R_i,f_{ij}\}$ will be denoted by $\invlim \{R_i\}_{i\in I}$.
    It
    can be understood as the set of $I$-tuples $(a_i)_{i\in I}\in\prod_{i\in I}R_i$ such that
    $f_{ij}(a_j)=a_i$ for all $i\leq j$ in $I$.

    \begin{lem} \label{RING:LM:basic-inv-lim-lem}
        Let $\{R_i,f_{ij}\}$ be an inverse system of rings and let $R=\invlim \{R_i\}$.
        Denote by $f_i$ the standard map from $R$ to $R_i$. Then:
        \begin{enumerate}
            \item[(i)] For all $n\in\N$, $\invlim \{\nMat{R_i}{n}\}_{i\in I}\cong  \nMat{\invlim \{R_i\}}{n}=\nMat{R}{n}$.
            \item[(ii)] For all $e\in\ids{R}$, $\invlim {\{e_iR_ie_i\}_{i\in I}}\cong eRe$ where $e_i=f_i(e)$.
        \end{enumerate}
    \end{lem}

    \begin{proof}
        This is straightforward.
    \end{proof}

    \begin{prp} \label{RING:PR:morita-preserved-properties}
        (i) Being semiprimary (resp.\ right perfect, semiperfect, semilocal, pro-semiprimary) is
        preserved under Morita equivalence.

        (ii) Let $R$ be a ring and $e\in \End(R)$. Then $R$ is semiprimary (resp.\ right perfect, semiperfect, semilocal)
        if and only if $eRe$ and $(1-e)R(1-e)$ are.
    \end{prp}

    \begin{proof}
        All statements regarding semiprimary, right perfect, semiperfect
        and semilocal rings are well known. The other statements follow from Lemma  \ref{RING:LM:basic-inv-lim-lem}.
    \end{proof}

    Part (ii) of Proposition \ref{RING:PR:morita-preserved-properties} does not hold for pro-semiprimary
    rings. For instance, take $R=\left\{\smallSMatII{a}{v}{0}{b}\where a,b\in\Z_p,~v\in\bigoplus_{i=1}^\infty\Z_p\right\}$
    (where $\Z_p$ are the $p$-adic integers) and let $e$ be the matrix unit $e_{11}$.

    \begin{thm}{\bf(Levitski).} \label{RING:TH:nil-subring}
        Let $R$ be a right noetherian ring. Then any nil subring of $R$ is nilpotent.
    \end{thm}

    Let $R$ be a ring. An element $a\in R$ is called \emph{right $\pi$-regular} (in $R$) if
    the right ideal chain $aR\supseteq a^2R\supseteq a^3R\supseteq\ldots$ stabilize.\footnote{
        This notion of $\pi$-regularity is sometimes called strong $\pi$-regularity.
    } If $a$ is both left and right $\pi$-regular we will say it is $\pi$-regular.
    A ring that all of its elements are right $\pi$-regular is called \emph{$\pi$-regular}.
    It was shown by Dischinger in \cite{Di76} that the latter property is actually left-right symmetric.
\rem{
    Our motivation in recalling these (somewhat uncommon) notions is the fact that $\pi$-regularity of elements is preserved upon
    moving to a semi-invariant subring. This observation, verified in
    Corollary \ref{RING:CR:pi-reg-transfers-to-semi-inv-subring} below, plays a crucial part in the main results
    of this paper.
}

    Since $\pi$-regularity is not preserved under Morita equivalence (see \cite{Ro89}), it is convenient to
    introduce the following notion: A ring $R$ is called $\pi_\infty$-regular\footnote{
        This property is sometimes called completely $\pi$-regular.
    } if $\nMat{R}{n}$ is $\pi$-regular for all $n$.

    \begin{prp} \label{RING:PR:pi-reg-hecke-algebra}
        (i) Let $R$ be a $\pi$-regular ($\pi_\infty$-regular) ring and $e\in\ids{R}$. Then $eRe$ is $\pi$-regular
        ($\pi_\infty$-regular).

        (ii) $\pi_\infty$-regularity is preserved under Morita equivalence.
    \end{prp}

    \begin{proof}
        (i) Assume $R$ is $\pi$-regular, let $e\in R$ and let $a=eae\in eRe$. By definition, there
        is $b\in R$ and $n\in \N$ such that $a^n=a^{n+1}b$. Multiplying by $e$ on the right we
        get $a^n=a^{n+1}ebe$, hence $a^n(eRe)=a^{n+1}(eRe)$.

        Assume $R$ is $\pi_\infty$-regular and let $e\in R$. Let $I$ denote
        identity matrix in $\nMat{R}{n}$. Then $(eI)\nMat{R}{n}(eI)=\nMat{eRe}{n}$.
        By the previous argument, the left-hand side is $\pi$-regular, hence we are through.

        (ii) We only need to check that $\nMat{R}{n}$ is $\pi_\infty$-regular for all $n\in \N$, which is obvious
        from the definition, and that $eRe$ is $\pi_\infty$-regular, which follows from (i).
    \end{proof}

    \begin{prp} \label{RING:PR:pi-reg-suff-cond}
        Let $R$ be a ring and let $N$ denote its prime radical (i.e.\
        the intersection of all prime ideals). Then $R$ is $\pi$-regular ($\pi_\infty$-regular)
        if and only if $R/N$ is.
    \end{prp}

    \begin{proof}
        See \cite[\S2.7]{Ro88}. (The argument is easily generalized to $\pi_\infty$-regular rings.)
    \end{proof}

    \begin{remark} \label{RING:RM:examples-of-pi-regular-rings}
        Any PI semilocal ring with nil Jacobson radical
        is $\pi_\infty$-regular (see \cite[Apx.]{Ro86}). However, there are semilocal rings with nil Jacobson
        radical that are not $\pi$-regular, see \cite{Ro89}.
    \end{remark}

    \begin{remark} \label{RING:RM:hirarchy-of-properties}
        We have the following implications:
        \[\textrm{right artinian}\derives\textrm{semiprimary}
        \derives\textrm{left/right perfect}\overset{\textrm{(\ref{RING:PR:Bass-Thm-P})}}{\derives}\textrm{$\pi_\infty$-regular}\derives\textrm{$\pi$-regular}\]
        However, all these notions coincide for right noetherian rings. Indeed, assume $R$ is $\pi$-regular and right noetherian
        and let $J=\Jac(R)$. Then
        $J$ is nil (see Lemma \ref{RING:LM:main-lemma-II}(i)), hence Theorem \ref{RING:TH:nil-subring} implies
        $J^n=0$ for some $n\in \N$. By Lemma \ref{RING:LM:main-lemma-II}(ii) below, $R$ is semiperfect and in particular
        $R/J$ is semisimple. As $R$ is right noetherian, the right $R/J$-modules $\{J^{i-1}/J^i\}_{i=1}^n$ are f.g.,
        hence their length as right $R$-modules is finite. It follows that $R_R$ has a finite length, so $R$ is right artinian.
    \end{remark}

    Throughout, we will  use implicitly the next lemma. Notice that it implies
    that being semiprimary (resp.\ right perfect, semiperfect, semilocal) passes to quotients.

    \begin{lem}
        Let $R$ be a semilocal ring. Then any surjective ring homomorphism $\vphi:R\to S$ satisfies $\vphi(\Jac(R))=\Jac(S)$.
    \end{lem}

    \begin{proof}
        $\vphi(\Jac(R))$ is an ideal of $\vphi(R)=S$ and $1+\vphi(\Jac(R))=\vphi(1+\Jac(R))\subseteq\vphi(\units{R})\subseteq \units{S}$,
        hence $\vphi(\Jac(R))\subseteq \Jac(S)$. On the other hand, $S/\vphi(\Jac(R))$ is a quotient of $R/\Jac(R)$
        which is semisimple. Therefore, $S/\vphi(\Jac(R))$ is semisimple, implying $\vphi(\Jac(R))\supseteq \Jac(S)$.
    \end{proof}

\section{Semi-Invariant Subrings}
\label{section:semi-inv}

    This section presents the basic properties of semi-invariant subrings.
    We begin by showing that for any ring the semi-invariant subrings are precisely the semi-centralizer subrings.

    \begin{prp} \label{RING:PR:semi-invariant-definition-prop}
        Let $R_0 \subseteq R$ be rings. The following are equivalent:
        \begin{enumerate}
            \item[(a)] There is a ring $S\supseteq R$ and a set $\Sigma\subseteq \End(S)$ such that $R_0=R^\Sigma$.
            \item[(b)] There is a ring $S\supseteq R$ and a subset $X\subseteq S$ such that $R_0=\Cent_R(X)$.
            \item[(c)] There is a ring $S\supseteq R$ and $\sigma\in \Aut(S)$
            such that  $\sigma^2=\id$ and $R_0=R^{\{\sigma\}}$.
            \item[(d)] There is a ring $S\supseteq R$ and an \emph{inner} automorphism $\sigma\in \Aut(S)$ such that $\sigma^2=\id$
            and $R_0=R^{\{\sigma\}}$.
            \item[(e)] There are rings $\{S_i\}_{i\in I}$ and ring homomorphisms $\psi^{(1)}_i,\psi^{(2)}_i:R \to S_i$
            such that $R_0=\{r\in R\,:\, \psi^{(1)}_i(r)=\psi^{(2)}_i(r),~\forall i\in I\}$.
        \end{enumerate}
    \end{prp}

    Note that condition (e) implies that the family of semi-invariant subrings is closed under intersection.

    \begin{proof}
        We prove (a)$\derives$(e)$\derives$(c)$\derives$(d)$\derives$(b)$\derives$(a).

        (a)$\derives$(e): Take $I=\Sigma$ and define $S_\sigma=S$, $\psi_\sigma^{(1)}=\sigma$ and $\psi_\sigma^{(2)}=\id_R$.

        (e)$\derives$(c): Let $\{S_i,\psi^{(1)}_i,\psi^{(2)}_i\}_{i\in I}$ be given. Without loss
        of generality we may assume that there is $i_0\in I$
        such that $S_{i_0}=R$ and $\psi_{i_0}^{(1)}=\psi_{i_0}^{(2)}=\id_R$.
        Define $S=\prod_{(i,j)\in I\times\{1,2\}} S_{ij}$ where
        $S_{ij}=S_i$ and let $\Psi:R\to S$ be given by
        \[\Psi(r)=\left(\psi_i^{(j)}(r)\right)_{(i,j)\in I\times\{1,2\}}\in S~.\]
        The existence of $i_0$  above implies $\Psi$ is injective. Let $\sigma\in\Aut(S)$ be the automorphism
        exchanging the $(i,1)$ and $(i,2)$ components of $S$ for all $i\in I$.
        Then
        one easily checks that $\sigma^2=\id$ and $\Psi(R)^{\{\sigma\}}=\Psi(R_0)$. We finish by identifying $R$ with $\Psi(R)$.

        (c)$\derives$(d): Let $S,\sigma$ be given and let $S'=S[x;\sigma]$ denote the ring
        of $\sigma$-twisted polynomials with (left) coefficients in $S$.
        Observe that $(x^2-1)\in\Cent(S')$ (since $\sigma^2=\id$), hence $S'':=S'/\ideal{x^2-1}$ is a free
        left $S$-module with basis $\{\quo{1},\quo{x}\}$
        (where $\quo{a}$ is the image of $a\in S'$ in $S''$). Let $\tau\in\Aut(S'')$ be conjugation by $\quo{x}$.
        Then $\tau^2=\id$ and $R^{\{\tau\}}=\{r\in R \suchthat \quo{x} r=r\quo{x}\}=\{r\in R \suchthat \sigma(r)=r\}=R^{\{\sigma\}}=R_0$.

        (d)$\derives$(b): This is a clear.

        (b)$\derives$(a): Let $S,X$ be given. Let $S'=S\dpar{t}$ be the ring of formal Laurent series $\sum_{n=k}^\infty a_nt^n$ ($k\in\Z$)
        with coefficients in $S$. The elements of $S'$ commuting with $X$ are precisely the elements that commute with $t^{-1}+X$ (as
        $t^{-1}$ is central in $S'$). However, it is easily seen that all elements in $t^{-1}+X$ are invertible. For all $x\in X$, let $\sigma_x
        \in\End(S')$ be  the inner automorphism of $S'$ given by conjugation with $t^{-1}+x$ and let $\Sigma=\{\sigma_x~|~x\in X\}$. Then
        ${R}^\Sigma = \Cent_R(t^{-1}+X)=\Cent_R(X)=R_0$.
    \end{proof}

    \begin{cor}
        Let $R,W$ be rings and let $\vphi:R\to W$ be a ring homomorphism. Assume $W_0\subseteq W$ is a semi-invariant
        subring of $W$. Then $\vphi^{-1}(W_0)$ is a semi-invariant subring of $R$.
    \end{cor}

    \begin{proof}
        By Proposition \ref{RING:PR:semi-invariant-definition-prop}(e), there are rings
        $\{S_i\}_{i\in I}$ and ring homomorphisms $\psi^{(1)}_i,\psi^{(2)}_i:W \to S_i$
        such that $W_0=\{r\in R\,:\, \psi^{(1)}_i(r)=\psi^{(2)}_i(r)~\forall i\in I\}$.
        Define $\vphi_i^{(n)}=\psi_i^{(n)}\circ \vphi$ and note that
        $\vphi^{-1}(W_0)=\{r\in R\suchthat \vphi(r)\in W_0\}=\{r\in R \suchthat \psi_i^{(1)}\vphi(r)=\psi_i^{(2)}\vphi(r)
        ~\forall i\in I\}=\{r\in R\suchthat \vphi_i^{(1)}(r)=\vphi_i^{(2)}(r)~\forall i\in I\}$.
    \end{proof}

    The equivalent conditions of Proposition \ref{RING:PR:semi-invariant-definition-prop}
    require the existence of some ambient ring. This leads to the following question:

    \begin{que}
        Is there an intrinsic definition of semi-invariant subrings?
    \end{que}

    \rem{By \emph{intrinsic} we mean a definition that does not assume the existence of \emph{some}
    mathematical object taken from a large category (e.g.\ an ambient ring).} Informally,
    we ask for a definition that would make it easy to show that a given subring is \emph{not} semi-invariant.
    \smallskip

    The next proposition is useful for producing examples of semi-invariant subrings.

    \begin{prp} \label{RING:PR:semi-inv-examples-pr}
        Let $R\subseteq S$ be rings and let $K$ be a central subfield of $S$. Then
        $R\cap K$ is a semi-invariant subring of $R$.
    \end{prp}

    \begin{proof}
        Let $S'=S\otimes_K S$ and define $\vphi_1,\vphi_2:S\to S'$
        by $\vphi_1(s)=s\otimes 1$ and $\vphi_2(s)=1\otimes s$.
        As $K$ is a central subfield, it is easy to
        check that $\{s\in S\suchthat \vphi_1(s)=\vphi_2(s)\}=K$, hence
        $\{s\in R\suchthat \vphi_1(s)=\vphi_2(s)\}=R\cap K$. We are done by
        Proposition \ref{RING:PR:semi-invariant-definition-prop}(e).
    \end{proof}

    \begin{cor} \label{RING:CR:semi-inv-subrings-of-a-field}
        Let $K$ be a field. Then the semi-invariant subrings of $K$ are precisely its subfields.
    \end{cor}

    \begin{proof}
        Any semi-invariant subring $R\subseteq K$ satisfies $\units{R}=
        R\cap\units{K}=R\setminus\{0\}$, hence it is a field. The
        converse follows from the last proposition.
    \end{proof}

    \begin{remark} \label{RING:RM:inv-subrings-of-a-field}
        If $K/L$ is an algebraic field extension, then $L$ is an invariant subring of $K$ if and only if
        $K/L$ is Galois.
    \end{remark}

    We finish this section by introducing two cases where semi-invariant subrings
    naturally appear.

    \begin{prp} \label{RING:PR:endo-ring-is-semi-invariant}
        Let $R\subseteq S$ be rings and let $M$ be a right $S$-module. Then $\End(M_S)$ is a semi-invariant
        subring of $\End(M_R)$.\footnote{
            This proposition is a refinement of \cite[Pr.\ 2.7]{FacHer06}, which
            asserts that under the same assumptions $\End(M_S)$ is a rationally closed subring of $\End(M_R)$.
        }
    \end{prp}

    \begin{proof}
        There is a ring homomorphism $\vphi:S^\op \to \End(M_{\Z})$ given
        by $\vphi(s^\op)(m)=ms$ for all $m\in M$. It is straightforward to check
        that $\End(M_S)=\Cent_{\End(M_{\Z})}(\im \vphi)$. As $\End(M_S)\subseteq \End(M_R)$, it follows that
        $\End(M_S)=\Cent_{\End(M_R)}(\im\vphi)$, hence $\End(M_S)$ is a semi-centralizer subring
        of $\End(M_R)$.
    \end{proof}

    \begin{prp} \label{RING:PR:exact-seq-lemma}
        Let $\CFont{A}$ be an abelian category and let $A\xrightarrow{f} B\xrightarrow{g} C\to 0$ be
        an exact sequence in $\CFont{A}$ such that for any $c\in \End(C)$ there
        are $b\in \End(B)$ and $a\in \End(A)$ with $cg=gb$ and $bf=fa$ (e.g.:
        if both $A$ and $B$ are projective, or if $B$ is projective and $f$ is injective).
        Then $\End(C)$ is isomorphic to a quotient of:
        \begin{enumerate}
            \item[(i)] A semi-invariant subring of $\End(A)\times\End(B)$.
            \item[(ii)] An invariant and a centralizer subring of $\End(A\oplus B)$, provided $f$ is injective.
        \end{enumerate}
    \end{prp}

    \begin{proof}
        Let $B_0=\im f=\ker g$.
        Define $R$ to be the subring of $\End(B)$ consisting of maps $b\in \End(B)$
        for which there is $a\in\End(A)$ with $bf=fa$. Then for all $b\in R$,
        $b(B_0)=b(\im f)=\im(fa)\subseteq \im f=B_0$. Therefore, there is \emph{unique} $c\in \End(C)$
        such that $cg=gb$. The map sending $b$ to $c$ is easily seen to be
        a ring homomorphism from $R$ to $\End(C)$ and the assumptions imply it is onto.
        Therefore, $\End(C)$ is a quotient of $R$.

        Let $S=\End(A\oplus B)$. We represent elements of $S$ as matrices $\smallSMatII{x}{y}{z}{w}$
        with $x\in \End(A),y\in \Hom(B,A),z\in\Hom(A,B),w\in \End(B)$. Let $D$
        denote the diagonal matrices in $S$ (i.e.\ $\End(A)\times \End(B)$)
        and let $W=\Cent_S(\smallSMatII{0}{f}{0}{0})$.
        Then
        for $a\in \End(A)$ and $b\in \End(B)$, $fa=bf$ if and only if
        $\smallSMatII{a}{0}{0}{b}\in W$. Define a ring homomorphism $\vphi:D\to \End(B)$ by $\vphi(\smallSMatII{x}{0}{0}{y})=y$.
        Then $\vphi(\Cent_D(\smallSMatII{0}{f}{0}{0}))=\vphi(D\cap W)=R$. It
        follows that $\End(C)$ is a quotient of $R$, which is a quotient of $D\cap W$, which is a semi-centralizer
        subring of $D=\End(A)\times \End(B)$. This settles (i). To see (ii), notice
        that if $f$ is injective, then $W$ consists of upper-triangular matrices, hence $\vphi$ can
        be extended to $W$, which is
        a centralizer and an invariant subring of $S$ since $W=\Cent_S(\smallSMatII{1}{f}{0}{1})$
        and $\smallSMatII{1}{f}{0}{1}\in\units{S}$.
    \end{proof}

\section{Properties Inherited by Semi-Invariant Subrings}
\label{section:main-thm}

    In this section we prove that being semiprimary (right perfect, semiperfect and $\pi_\infty$-regular,
    semiperfect and $\pi$-regular) passes to semi-invariant subrings. We also present
    a supplementary result for algebras.
    \smallskip

    Our first step is introducing an equivalent condition for $\pi$-regularity of elements of a ring.

    \begin{lem} \label{RING:LM:main-lemma-I}
        Let $R$ be a ring and let $a\in R$ be a $\pi$-regular element. Define:
        \[\begin{array}{cc}
        A=\bigcap_{k=1}^\infty a^kR\,,\quad & B=\bigcup_{k=1}^\infty \annr a^k, \\
        A'=\bigcap_{k=1}^\infty Ra^k,\quad & B'=\bigcup_{k=1}^\infty \annl a^k.
        \end{array}\]
        Then there is $e\in \ids{R}$
        such that $A=eR$, $B=fR$, $A'=Re$ and $B'=Rf$ where $f:=1-e$.
        In particular, $R_R=A\oplus B$ and ${}_RR=A'\oplus B'$.
    \end{lem}

    \begin{proof}
        Let $n\in \N$ be such that $a^nR=a^kR$ and $Ra^n=Ra^k$ for all $k\geq n$.
        Notice that this implies $\annr a^n=\annr a^k$ and $\annl a^n=\annr a^k$ for all $k\geq n$.

        We begin by showing $R_R=A\oplus B$. That ${}_RR=A'\oplus B'$ follows by symmetry.
        The argument is similar to the proof
        of Fitting's Lemma (see \cite[\S2.9]{Ro88}):
        Let $r\in R$. Then $a^{n}r\in a^nR= a^{2n}R$, hence there is $s\in R$ with
        $a^{n}r=a^{2n}s$. Observe that $a^n(r-a^ns)=0$ and $a^ns\in a^nR$, so $r=a^ns+(r-a^ns)\in A+ B$. Now
        suppose $r\in A\cap B$. Then $r=a^ns$ for some $s\in R$. However, $r\in B=\annr a^n$ implies $s\in \annr a^{2n}=\annr a^{n}$, so
        $r=a^ns=0$.

        Since $R_R=A\oplus B$ there is $e\in R$ such that $e\in A$ and $f:=1-e\in B$. It is well known
        that in this case $e^2 = e$, $A=eR$ and $B=fR$.
        This implies $B'=\annl a^n=\annl a^nR=\annl eR=Rf$ and
        $A'=Ra^n\subseteq \annl\annr Ra^n=\annl \annr a^n=\annl fR=Re$. As ${}_RR=A'\oplus B'=Re\oplus Rf$, we
        must have have $A'=Re$.
    \end{proof}

    \begin{prp} \label{RING:PR:pi-reg-equiv-cond}
        Let $R$ be a ring and $a\in R$. Then $a$ is $\pi$-regular $\iff$ there is $e\in \ids{R}$
        such that
        \begin{enumerate}
            \item[(A)] $a=eae+faf$ where $f:=1-e$.
            \item[(B)] $eae$ is invertible in $eRe$.
            \item[(C)] $faf$ is nilpotent.
        \end{enumerate}
        In this case, the idempotent $e$ is uniquely determined by $a$.
    \end{prp}

    \begin{proof}
        Assume $a$ is $\pi$-regular and let $e,f,A,B,A',B',n$ be as in Lemma \ref{RING:LM:main-lemma-I}.
        Then $ae\in aeR=aA=a(a^nR)=a^{n+1}R=A=eR$
        and $af\in aB=a\annr a^n\subseteq \annr a^n=B=fR$. Therefore, $ae=eae$ and $af=faf$, hence
        $a=ae+af=eae+faf$.
        This implies $a^k=(eae)^k+(faf)^k$ for all $k\in \N$.
        As $a^n\in eR$, we have $a^n=ea^n$, hence
        $(eae)^n+(faf)^n=a^n=e(eae)^n+e(faf)^n=(eae)^n$
        which implies $(faf)^n=0$.
        In particular, for all $k\geq n$, $a^k=(eae)^k+(faf)^k=(eae)^k$. Since $e\in a^nR$,
        there is $x\in R$ such that $e=a^nx=(eae)^nx$. Multiplying by $e$ on the right yields $e=(eae)((eae)^{n-1}xe)$,
        hence $eae$ is right invertible in $eRe$. By symmetry, $eae$ is left also left invertible
        in $eRe$, hence we conclude that $e$ satisfies (A)--(C).

        Now assume there is $e\in\ids{R}$ satisfying (A)--(C) and
        let $b$ be the inverse of $a$ in $eRe$.
        Then $a^k=(eae)^k+(faf)^k$ for all $k\in \N$. Condition (C) now implies there is $n\in\N$
        such that $a^k=(eae)^k$ for all $k\geq n$.
        Therefore, for all $k\geq n$,
        $a^n=(eae)^n=(eae)^kb^{k-n}=a^kb^{k-n}\in a^kR$ implying
        $a^nR=a^kR$. By symmetry, $Ra^n=Ra^k$ for all $k\geq n$,
        so $a$ is $\pi$-regular.

        Finally, assume that $e,e'\in\ids{R}$ satisfy conditions (A)--(C) and let $f=1-e$, $f'=1-e'$.
        By the previous paragraph $a$ is $\pi$-regular, hence Lemma \ref{RING:LM:main-lemma-I} implies
        $R=A\oplus B$ where $A=\bigcap_{k=1}^\infty a^kR$
        and $B=\bigcup_{k=1}^\infty \annr a^kR$.
        Let $b$ be the inverse
        of $eae$ in $eRe$ and let $n\in\N$ be such that $(faf)^n=0$.
        Then $e=(eae)^kb^k=a^kb^k\in a^kR$ for all $k\geq n$, hence $e\in A$, and $a^nf=(eae)^nf=0$, hence $f\in B$.
        Similarly, $e'\in A$ and $f'\in B$. It follows that $e,e'\in A$ and $f,f'\in B$. Since $1=e+f=e'+f'$
        and $R=A\oplus B$, we must have $e=e'$.
    \end{proof}

    Let $R,a$ be as in Proposition \ref{RING:PR:pi-reg-equiv-cond}. Henceforth, we call the unique
    idempotent $e$ satisfying conditions (A)--(C) the \emph{associated idempotent} of $a$ (in $R$).

    \begin{cor} \label{RING:CR:pi-reg-transfers-to-semi-inv-subring}
        (i) Let $R$ be a ring, $R_0\subseteq R$ a semi-invariant subring and let $a\in R_0$ be $\pi$-regular
        in $R$.
        Then $a$ is $\pi$-regular in $R_0$.

        (ii) A semi-invariant subring of a $\pi$-regular ($\pi_\infty$-regular) ring is $\pi$-regular ($\pi_\infty$-regular).
    \end{cor}

    \begin{proof}
        (i)
        Let $S\supseteq R$ and $\Sigma\subseteq\End(S)$ be such that $R_0=R^\Sigma$ and let $a\in R_0$
        be $\pi$-regular in $R$. Let $e$ be the associated idempotent of $a$ in $R$.
        Then $e$ is clearly the associated idempotent of $a$ in $S$ (hence $a$ is $\pi$-regular in $S$).
        However, it is straightforward to check that $\sigma(e)$
        satisfies conditions (A)--(C) (in $S$) for all $\sigma\in \Sigma$
        (since $\sigma(a)=a$), so the uniqueness of $e$ forces $e\in S^\Sigma\cap R=R^\Sigma=R_0$.
        Therefore, $a$ is $\pi$-regular in $R_0$.

        (ii) The $\pi$-regular case follows from (i). The $\pi_\infty$-regular case
        follows once noting that if $R_0$ is a semi-invariant subring of $R$, then
        $\nMat{R_0}{n}$ is a semi-invariant subring of $\nMat{R}{n}$ for all $n\in \N$.
    \end{proof}

    \begin{lem} \label{RING:LM:main-lemma-II}
        Let $R$ be a $\pi$-regular ring. Then:
        \begin{enumerate}
            \item[(i)] $\Jac(R)$ is nil.
            \item[(ii)] $R$ is semiperfect $\iff$ $R$ does not contain an infinite set of orthogonal idempotents.
        \end{enumerate}
    \end{lem}

    \begin{proof}
        For $a\in R$, let $e_a$ denote the associated idempotent of $a$ and let $f_a=1-e_a$.

        (i) Let $a\in \Jac(R)$ and let
        $b$ be the inverse of $e_aae_a$ in $e_aRe_a$. Then $e_a=b(e_aae_a)\in \Jac(R)$, hence $e_a=0$, implying $a=f_aaf_a$ is
        nilpotent.

        (ii) That $R$ is semiperfect clearly implies $R$ does not contain an infinite set of
        orthogonal idempotents, so assume the converse.
        Let $a\in R$. Observe that if $e_a=0$ then $a$ is nilpotent and if $e_a=1$ then
        $a$ is invertible. Therefore, if $e_a\in\{0,1\}$ for all $a\in R$, then $R$ is local and in particular,
        semiperfect.

        Assume there is $a\in R$ with $e:=e_a\notin \{0,1\}$. We now apply an inductive
        argument to deduce that $eRe$ and $(1-e)R(1-e)$ are semiperfect,
        thus proving $R$ is semiperfect (by Proposition \ref{RING:PR:morita-preserved-properties}).
        The induction process must stop because otherwise there is a sequence of idempotents $\{e_k\}_{k=0}^\infty\subseteq R$
        such that $e_k\in e_{k-1}Re_{k-1}$ and $e_k\notin \{0,e_{k-1}\}$. This implies $\{e_{k-1}-e_k\}_{k=1}^\infty$ is an infinite
        set of non-zero orthogonal idempotents, which cannot exist by our assumptions.
    \end{proof}

    \begin{lem} \label{RING:LM:main-lemma-III}
        Let $R_0\subseteq R$ be rings. If $R$ is semiperfect and both $R_0$ and $R$ are $\pi$-regular,
        then $R_0$ is semiperfect and $\Jac(R_0)^n\subseteq \Jac(R)$ for some $n\in\N$.
        If in addition $R$ is semiprimary (right perfect), then so is $R_0$.
    \end{lem}

    \begin{proof}
        By Lemma \ref{RING:LM:main-lemma-II}(ii), $R$ does not contain an infinite
        set of orthogonal idempotents. Therefore, this also applies to $R_0$, so the same
        lemma implies $R_0$ is semiperfect. Let $\vphi$ denote the standard
        projection from $R$ to $R/\Jac(R)$.
        By Lemma \ref{RING:LM:main-lemma-II}(i), $\vphi(\Jac(R_0))$ is nil.
        Therefore, by Theorem \ref{RING:TH:nil-subring} (applied to $\vphi(R)$, which is semisimple), $\vphi(\Jac(R_0))$ is nilpotent,
        hence there is $n\in\N$ such that $\Jac(R_0)^n\subseteq \Jac(R)$.
        If moreover $R$ is semiprimary (right perfect), then $\Jac(R)$ is nilpotent (right T-nilpotent).
        The inclusion $\Jac(R_0)^n\subseteq \Jac(R)$ then
        implies $\Jac(R_0)$ is nilpotent (right T-nilpotent), so $R_0$ is semiprimary (right perfect).
    \end{proof}

    \begin{thm} \label{RING:TH:main-res-general}
        Let $R$ be a ring and let $R_0$ be a semi-invariant subring of $R$.
        If $R$ is semiprimary (resp.\ right perfect, semiperfect and $\pi_\infty$-regular, semiperfect and $\pi$-regular), then so is $R_0$.
        In addition, there is $n\in\N$
        such that $\Jac(R_0)^n\subseteq \Jac(R)$.
    \end{thm}

    \begin{proof}
        Recall that being right perfect implies being $\pi$-regular by Proposition \ref{RING:PR:Bass-Thm-P}.
        Given that, the theorem follows from Corollary \ref{RING:CR:pi-reg-transfers-to-semi-inv-subring} and
        Lemma \ref{RING:LM:main-lemma-III}.
    \end{proof}

    \begin{cor}\label{RING:CR:main-cr:I}
        Let $R\subseteq S$ be rings and let $M$ be a right $S$-module.
        If $\End(M_R)$ is semiprimary (resp.\ right perfect, semiperfect and $\pi_\infty$-regular, semiperfect and $\pi$-regular),
        then so is $\End(M_S)$
        and there exists $n\in \N$
        such that $\Jac(\End(M_S))^n\subseteq \Jac(\End(M_R))$.
    \end{cor}

    \begin{proof}
        This follows from the last theorem and Proposition \ref{RING:PR:endo-ring-is-semi-invariant}.
    \end{proof}

    \begin{remark}
        Camps and Dicks proved in \cite{CaDi93} that
        a \emph{rationally closed} subring of a semilocal ring
        is semilocal, thus implying the semilocal analogues of Theorem \ref{RING:TH:main-res-general} and
        Corollary \ref{RING:CR:main-cr:I}, excluding the part regarding the Jacobson radical
        (which indeed fails in this case; see Example \ref{RING:EX:non-semi-inv-subring}).
        In fact, the semilocal analogue of Corollary \ref{RING:CR:main-cr:I} was noticed in \cite[Pr.\ 2.7]{FacHer06}.
        However, we cannot use this analogue with the Krull-Schmidt Theorem (as we do in Section \ref{section:applications}
        with our results) because
        modules with semilocal endomorphism ring need not have a Krull-Schmidt decomposition, as shown in  \cite{FHL95} and \cite{BFR01}.

        Nevertheless, as there are plenty of weaker Krull-Schmidt theorems for modules
        that do not require $\End(M_R)$ to be semiperfect  (mainly due to Facchini et al.; e.g.\ \cite{BFR01},
        \cite{FaEcKo10}), it might be that if $M,R,S$ are as in Corollary \ref{RING:CR:main-cr:I}
        and $\End(M_R)$ is merely semiperfect, then $M$ has a Krull-Schmidt decomposition over $S$
        (despite the fact $\End(M_S)$ need not be semiperfect).
        To the author's best knowledge, this topic is still open.
    \end{remark}

\rem{
    \begin{remark} \label{RING:RM:main-cr:extension}
        Corollary \ref{RING:CR:main-cr:I} remains true if we replace semiprimary with
        semiperfect-and-$\pi$-regular (-$\pi_\infty$-regular). The proof is the same.
    \end{remark}
}

    We finish this section with a supplementary result for algebras.

    \begin{prp}
        Let $R\subseteq S$ be rings and $\Sigma\subseteq \End(S)$. Assume there
        is a division ring $D\subseteq R$ such that $\sigma(D)\subseteq D$
        for all $\sigma\in\Sigma$.
        Then $\dim {}_{D^\Sigma}R^\Sigma\leq \dim {}_DR$.
    \end{prp}

    \begin{proof}
        Consider the left $D$-vector space $V=DR^\Sigma$. Let $\{v_i\}_{i\in I}\subseteq R^\Sigma$
        be a left $D$-basis for $V$. We claim that $\{v_i\}_{i\in I}$ is a left $D^\Sigma$-basis
        for $R^\Sigma$. Indeed, let $v\in R^\Sigma$. Then  there are
        unique $\{d_i\}_{i\in I}\subseteq D$ (almost all $0$)
        such that $v=\sum_i d_iv_i$. However,
        for all $\sigma\in \Sigma$, $v=\sigma(v)=\sum_i \sigma(d_i)v_i$, so
        $\sigma(d_i)=d_i$ for all $i\in I$, hence $v\in \sum_{i\in I}D^\Sigma v_i$.
        Therefore, $\dim {}_{D^\Sigma}R^\Sigma = \dim {}_DV\leq \dim {}_DR$.
    \end{proof}

    \begin{remark}
        An \emph{invariant} subring of a f.d.\ algebra need not be left nor right artinian,
        even when invariants are taken w.r.t.\ to the action of a finite cyclic group.
        This was demonstrated by Bjork in \cite[\S2]{Bj71B}.
        In particular, the assumption $\sigma(D)\subseteq D$ for all $\sigma\in\Sigma$ in the last proposition
        is essential. However, Bjork also proved that if $\Sigma$ is a \emph{finite group} acting on a f.d.\ algebra over
        a \emph{perfect} field, then the invariant subring (w.r.t.\ $\Sigma$) is artinian; see \cite[Th. 2.4]{Bj71B}.
        For a detailed discussion about when a subring of an artinian ring is artinian, see
        \cite{Bj71B} and \cite{Bj71A}.
    \end{remark}

\section{Topologically Semi-Invariant Subrings}
\label{section:top-rings}

    In this section we specialize the notions of
    semi-invariance and and $\pi$-regularity to certain topological rings.
    As a result we obtain a topological analogue of Theorem \ref{RING:TH:main-res-general}
    (Theorem \ref{RING:TH:main-res-for-quasi-pi-reg-rings}),
    which is used to prove that \emph{topologically semi-invariant}
    subrings of
    semiperfect pro-semiprimary rings are semiperfect and pro-semiprimary (Theorem \ref{RING:TH:pro-semiprim-full-result}).
    Note that once
    restricted to
    \emph{discrete} topological rings, some of the results of this section
    reduce to results from the previous sections. However, the latter are
    not superfluous since we will rely on them.
    For a general reference about topological rings, see \cite{Wa93}.

    \begin{dfn}
        A topological ring $R$ is called \emph{linearly topologized} (abbreviated: LT) if it admits
        a local basis (i.e.\ a basis of neighborhoods of $0$) consisting of two-sided ideals. In
        this case the topology on $R$ is called \emph{linear}.
    \end{dfn}

    Let us set some general notation: For any topological ring $R$,  let $\calI_R$ denote
    its set of open ideals. Then $R$ is LT if and only if $\calI_R$ is a local basis.
    We use $\cHom$ ($\cEnd$)
    to denote continuous homomorphisms (endomorphisms).
    The category of \emph{Hausdorff} linearly topologized rings will be denoted
    by $\LTR$, where $\Hom_{\LTR}(A,B)=\cHom(A,B)$ for all $A,B\in \LTR$.\footnote{
        The subscript ``$2$'' in $\LTR$ stands for the second separation axiom $\mathrm{T}_2$ (i.e.\ being Hausdorff).
    }
    A subring of a topological ring
    is assumed to have the induced topology.
    In particular, if $R\in \LTR$ then so is any subring of $R$.
    The following facts will be used freely throughout the paper. For proofs, see \cite[\S3]{Wa93}.
    \begin{enumerate}
        \item[(1)] Let $(G,+)$ be an abelian topological group and let $\calB$ be a local
        basis of $G$. Then for any subset $X\subseteq G$, $\overline{X}=\bigcap_{U\in\calB}(X+U)$.
        \item[(2)] Under the pervious assumptions, $G$ is Hausdorff $\iff$ $\overline{\{0\}}=\bigcap_{U\in\calB}U=\{0\}$.
        \item[(3)] Given a ring $R$ and a filter base of ideals $\calB$, there exists
        a unique ring topology on $R$ with local basis $\calB$. This topology makes $R$ into an LT ring.
    \end{enumerate}

    \begin{example} \label{RING:EX:top-rings}
        (i) Any ring assigned with the discrete topology is LT.

        (ii) $\Z_p$ (with the $p$-adic topology) is LT but $\Q_p$ is not.

        (iii) Let $R$ be an LT ring and let $n\in\N$. We make $\nMat{R}{n}$ into an LT ring
        by assigning it the unique ring topology with local basis $\{\nMat{I}{n}\where I\in\calI_R\}$.

        (iv) If $R$ is LT and $e\in \ids{R}$, then $eRe$ is LT w.r.t.\ the induced topology.

        (v) Let $\{R_i\}_{i\in I}$ be LT rings. Then $\prod_{i\in I}R_i$ is LT w.r.t.\
        the product topology.

        (vi) Let $R$ be an inverse limit of LT rings $\{R_i\}_{i\in I}$,
        with the topology induced from the product topology on $\prod_{i\in I}R_i$.\footnote{
            With this topology $R$ is indeed the inverse limit of $\{R_i\}_{i\in I}$ in category of topological rings, i.e.\
            it admits the required universal property.
        }
        Then by (v) $R$ is LT.

        (vii) If $R$ is LT and $J\idealof R$, then $R/J$ with the quotient topology is LT. Indeed, $\{I/J\where J\subseteq I\in\calI_R\}$
        is a local basis for that topology. The ring $R/J$ is Hausdorff if and only if $J$ is closed, and discrete if and only if $J$ is open.
    \end{example}

    The last example implies that $\LTR$ is closed under products, inverse limits and forming matrix rings (with the appropriate
    topologies). We will say that a property $\calQ$ of LT rings is preserved under Morita equivalence if
    whenever $R\in\LTR$ has $\calQ$, then so does $\nMat{R}{n}$ and $eRe$ for $e\in\ids{R}$ s.t. $eR$ is a
    progenerator.\footnote{
        Caution: There is a notion of Morita equivalence for (right) LT rings, but we will not use it in this paper; see \cite{Gr88}
        and related articles.
    }

    \begin{dfn}
        Let $R\in \LTR$. A subring $R_0\subseteq R$ is called a \term{topologically semi-invariant}
        (abbrev.: T-semi-invariant) subring
        if there is $R\subseteq S\in\LTR$ and a set $\Sigma\subseteq \cEnd(S)$
        such that $R_0=R^\Sigma$. The subring $R_0$ is called a \term{topologically semi-centralizer}
        (abbrev.: T-semi-centralizer)
        subring if there is $R\subseteq S\in\LTR$ and a set $X\subseteq S$ such
        that $R_0=\Cent_R(X)$.
    \end{dfn}

    A T-semi-invariant subring is always closed. In addition,
    there is an analogue of Proposition \ref{RING:PR:semi-invariant-definition-prop} for T-semi-invariant rings.

    \begin{prp} \label{RING:PR:T-semi-invariant-definition-prop}
        Let $R_0$ be a subring of $R\in \LTR$. The following are equivalent:
        \begin{enumerate}
            \item[(a)] There is $R\subseteq S\in \LTR$ and a set $\Sigma\subseteq \cEnd(S)$ such that $R_0=R^\Sigma$.
            \item[(b)] There is $R\subseteq S\in \LTR$ and a subset $X\subseteq S$ such that $R_0=\Cent_R(X)$.
            \item[(c)] There is $R\subseteq S\in \LTR$ and $\sigma\in \cAut(S)$
            with  $\sigma^2=\id$ and $R_0=R^{\{\sigma\}}$.
            \item[(d)] There is $R\subseteq S\in \LTR$ and an \emph{inner} automorphism $\sigma\in \cAut(S)$ such that $\sigma^2=\id$
            and $R_0=R^{\{\sigma\}}$.
            \item[(e)] There are LT Hausdorff rings $\{S_i\}_{i\in I}$ and
            continuous ring homomorphisms $\psi^{(1)}_i,\psi^{(2)}_i:R \to S_i$
            such that $R_0=\{r\in R\,:\, \psi^{(1)}_i(r)=\psi^{(2)}_i(r)~\forall i\in I\}$.
        \end{enumerate}
    \end{prp}

    \begin{proof}
        This is essentially the proof of Proposition \ref{RING:PR:semi-invariant-definition-prop},
        but we need to endow the rings constructed throughout the proof with topologies making
        them into LT Hasudorff rings that contain $R$ as a topological ring. This is briefly
        done below; the details are left to the reader.

        (e)$\derives$(c): Endow $S=\prod_{(i,j)\in I\times \{1,2\}}S_{ij}$ with the product topology.

        (c)$\derives$(d): Observe that
        $\calB=\{I\cap \sigma(I)\where I\in\calI_S\}$ is a local basis of $S$ and $\sigma(J)=J$ for all $J\in\calB$.
        Assign $S'=S[x;\sigma]$ the unique ring topology with local basis $\{J[x;\sigma]\where J\in\calB\}$,
        where $J[x;\sigma]$ denotes the set of polynomials with (left) coefficients in $J$, and give $S''=S'/\ideal{x^2-1}$
        the quotient topology.

        (b)$\derives$(a): Give $S\dpar{t}$ the unique ring topology with local basis $\{I\dpar{t}\where I\in\calI_S\}$,
        where $I\dpar{t}$ denotes the set of polynomials with coefficients in $I$.
    \end{proof}

    We now generalize the notion of $\pi$-regularity for topological rings.
    Our definition is inspired by Proposition \ref{RING:PR:pi-reg-equiv-cond}.

    \begin{dfn}
        Let $R\in\LTR$ and $a\in R$. The element $a$ is called \term{quasi-$\pi$-regular} in $R$ if
        there is an idempotent $e\in \ids{R}$ such that:
        \begin{enumerate}
            \item[(A)] $a=eae+faf$ where $f:=1-e$.
            \item[(B)] $eae$ is invertible in $eRe$.
            \item[(C$'$)] $(faf)^n\xrightarrow{n\to\infty} 0$ (w.r.t.\ the topology on $R$).
        \end{enumerate}
        Call $R$ \term{quasi-$\pi$-regular} if all its elements are quasi-$\pi$-regular.
    \end{dfn}

    Since we only consider LT rings, condition (C$'$) means that for any $I\in\calI_R$ there is $n\in\N$ such that $(faf)^n\in I$.
    This implies that quasi-$\pi$-regularity coincide with $\pi$-regularity for
    discrete topological rings (take $I=\{0\}$) and that if $a$ is quasi-$\pi$-regular in $R$
    then $a+I$ is $\pi$-regular in $R/I$ for all $I\in\calI_R$. In particular, if $R$ is
    quasi-$\pi$-regular then $R/I$ is $\pi$-regular.\rem{ Henceforth, these facts will be used freely.}
    We will call the idempotent $e$ satisfying conditions (A),(B) and (C$'$) the \emph{associated idempotent}
    of $a$. The following lemma shows that it is unique.

    \begin{lem} \label{RING:LM:uniqueness-lemma}
        Let $R\in\LTR$ and $a\in R$ a quasi-$\pi$-regular element. Then
        the idempotent $e$ satisfying conditions (A),(B) and (C\,$'$\!) is uniquely determined by $a$.
    \end{lem}

    \begin{proof}
        Assume both $e$ and $e'$ satisfy conditions (A), (B), (C$'$) and
        let $I\in\calI_R$. Then $e+I$ and $e'+I$ are associated idempotents of $a+I$ in $R/I$,
        hence $e+I=e'+I$, or equivalently $e-e'\in I$.
        It follows that $e-e'\in\bigcap_{I\in\calI_R}I=\{0\}$ (since $R$ is Hausdorff), so $e=e'$.
    \end{proof}

    \begin{remark} \label{RING:RM:remark-after-uniqueness-lemma}
        (i) In the assumptions of the previous lemma it is also possible to show
        that $eR=\bigcap_{n=1}^\infty a^nR$
        and $(1-e)R=\{r\in R\suchthat a^nr\xrightarrow{n\to\infty} 0\}$.

        (ii) If we do not restrict to LT Hausdorff rings, the associated idempotent need not be unique.
        For example, in $\Q_p$ both $0$ and $1$ are associated idempotents of $p$. (It is not known
        if Lemma \ref{RING:LM:uniqueness-lemma} holds under the assumption that $R$ is \emph{right LT}, i.e.\
        has a local basis of right ideals).

        (iii) If one assigns a semiperfect ring $R$ with $\bigcap_{n=1}^{\infty}\Jac(R)^n=\{0\}$ the $\Jac(R)$-adic
        topology, then $R$ becomes a Hausdorff LT ring and for any $a\in R$ there is an idempotent $e$ satisfying
        conditions (B) and (C$'$) (but such $e$ need not be
        unique even when $R$ is simple). However, condition (A) might be impossible to satisfy for some $a$.
        Indeed, the ring $R$ constructed in
        example \ref{RING:EX:inv-subring-semiperfect-case-I} below, which is isomorphic to $\nMat{\Z_{\ideal{3}}}{4}$,
        is a semiperfect ring having no ring topology making it into a quasi-$\pi$-regular Hausdorff LT ring.
        As $\Z_{\ideal{3}}$ is quasi-$\pi$-regular w.r.t.\ the $3$-adic topology (since it is local),
        it follows that quasi-$\pi$-regularity is not preserved under Marita equivalence. (This also
        follows from the comment before Proposition \ref{RING:PR:pi-reg-hecke-algebra}.)
    \end{remark}

    It light of the last remark, it is convenient to call an LT Hausdorff ring $R$ \emph{quasi-$\pi_\infty$-regular} if $\nMat{R}{n}$
    is quasi-$\pi$-regular for all $n$. This property is preserved under Morita
    equivalence and turns out to be related with
    \emph{Henselianity} (see Section \ref{section:modules}).
    However, to avoid cumbersome notation, we will not mention it in this section.
    All statements henceforth can be easily seen to hold once replacing (quasi-)$\pi$-regular
    with (quasi-)$\pi_\infty$-regular.

    \begin{cor} \label{RING:CR:quasi-pi-reg-transfers-to-semi-inv-subring}
        (i) Let $R\in\LTR$, let $R_0$ be a T-semi-invariant subring of $R$ and let $a\in R_0$ be quasi-$\pi$-regular
        in $R$.
        Then $a$ is quasi-$\pi$-reuglar in $R_0$.

        (ii) A T-semi-invariant subring of a quasi-$\pi$-regular ring is quasi-$\pi$-regular.
    \end{cor}

    \begin{proof}
        This is similar to the proof of Corollary \ref{RING:CR:pi-reg-transfers-to-semi-inv-subring}.
    \end{proof}

    \begin{lem} \label{RING:LM:main-lemma-II-quasi-pi-reg}
        Let $R\in\LTR$ be quasi-$\pi$-regular. Then:
        \begin{enumerate}
            \item[(i)] For all $a\in \Jac(R)$, $a^n\xrightarrow{n\to\infty}0$. (That is, $\Jac(R)$ is ``topologically nil'').
            \item[(ii)] $R$ is semiperfect $\iff$ $R$ does not contain an infinite set of orthogonal idempotents.
        \end{enumerate}
    \end{lem}

    \begin{proof}
        (i) Let $I\in\calI_R$. Then $a+I\in(\Jac(R)+I)/I\subseteq \Jac(R/I)$,
        so by Lemma \ref{RING:LM:main-lemma-II}(i) applied to $R/I$ (which is $\pi$-regular),
        there is $n\in\N$ such that $a^n\in I$.

        (ii) We only show the non-trivial implication.
        For $a\in R$, let $e_a$ denote the associated idempotent of $a$.
        Note that $e_a=1$ implies $a\in\units{R}$ and $e_a=0$ implies $a^n\xrightarrow{n\to\infty}0$.

        Assume $e_a\in\{0,1\}$ for all $a\in R$. We claim $R$ is local. This is clear if $R=\{0\}$.
        Otherwise, let $a\in R$ and assume by contradiction that $e_a=e_{1-a}=0$.
        Let $R\neq I\in\calI_R$ (here we need $R\neq \{0\}$). Then there is $n\in \N$
        such that $a^n,(1-a)^n\in I$, implying $(1-a^n)^n=(1-a)^n(1+a+\dots+a^{n-1})^n\in I$.
        We can write $1=(1-a^n)^n+a^nh(a)$ for some $h(x)\in\Z[x]$, thus
        getting $1\in I$, in contradiction to the assumption $I\neq R$. Therefore, one of $e_a$, $e_{1-a}$ is $1$, hence
        one of $a$, $1-a$ is invertible.

        Now assume there is $a\in R$ with $e:=e_a\notin\{0,1\}$. Then we can induct on $eRe$
        and $(1-e)R(1-e)$ as in the proof of Lemma \ref{RING:LM:main-lemma-II}(ii).
        However, we need to verify that $eRe$ is quasi-$\pi$-regular (w.r.t.\ the induced topology).
        Let $b\in eRe$. It
        enough to show $e_b\in eRe$, i.e.\ $e_b=ee_be$. As $R\in\LTR$, this is equivalent to $e_b+I=ee_be+I$ for
        all $I\in \calI_R$. Indeed, since $R/I$ is $\pi$-regular, so is $e(R/I)e$ (by Proposition \ref{RING:PR:pi-reg-hecke-algebra}(i)),
        hence $b+I$ has an associated idempotent $\veps\in e(R/I)e$.
        However, it easy to see that $\veps$ is also the associated idempotent of $b+I$ in $R/I$, so necessarily $\veps=e_b+I$. As
        $\veps=(e+I)\veps (e+I)$,
        it follows that $e_b+I=ee_be+I$.
    \end{proof}

    We can now state and prove a T-semi-invariant analogue of Theorem \ref{RING:TH:main-res-general}.

    \begin{thm} \label{RING:TH:main-res-for-quasi-pi-reg-rings}
        Let $R_0$ be a T-semi-invariant subring of a semiperfect and quasi-$\pi$-regular ring $R\in\LTR$.
        Then $R_0$ is semiperfect and quasi-$\pi$-regular and
        there is $n\in \N$ such that $\Jac(R_0)^n\subseteq \Jac(R)$.
    \end{thm}

    \begin{proof}
        That $R_0$ is quasi-$\pi$-regular and semiperfect follows form Corollary \ref{RING:CR:quasi-pi-reg-transfers-to-semi-inv-subring}
        and Lemma \ref{RING:LM:main-lemma-II-quasi-pi-reg}(ii).
        Now let $I\in\calI_R$. Then both $R/I$ and $(R_0+I)/I$ are semiperfect and $\pi$-regular
        (since $(R_0+I)/I\cong R_0/(R_0\cap I)$ and $R_0\cap I$ is open in $R_0$).
        Therefore, by Lemma \ref{RING:LM:main-lemma-III}, there is $n_I\in\N$ such
        that $\Jac((R_0+I)/I)^{n_I}\subseteq \Jac(R/I)^{n_I}$.
        As $\Jac(R/I)=(\Jac(R)+I)/I$,
        this implies $\Jac(R_0)^{n_I}\subseteq \Jac(R)+I$.
        However, $(R/I)/(\Jac(R/I))\cong R/(\Jac(R)+I)$ is a quotient of $R/\Jac(R)$ which is semisimple, hence
        the index of nilpotence of any of its subsets is bounded (when finite) by $\mathrm{length}(R/\Jac(R))$.\footnote{
            Actually, the index of nilpotence is bounded in any right noetherian ring $R$. Indeed,
            the prime radical of $R$, denoted $N$, is nilpotent and $R/N$ is a semiprime
            Goldie ring. Therefore, by Goldie's Theorem, $R/N$ embeds in a
            semisimple ring and thus has a bounded index of nilpotence.
        }
        Therefore, there is $n\in\N$ such that for all $I\in\calI_R$, $\Jac(R_0)^n\subseteq \Jac(R)+I$ or
        equivalently, $\Jac(R_0)^n\subseteq\bigcap_{I\in\calI_R}(\Jac(R)+I)=\overline{\Jac(R)}$.
        Thus, we are done by the following lemma.
    \end{proof}

    \begin{lem} \label{RING:LM:closed-jacobson-radical}
        Let $R\in\LTR$ be quasi-$\pi$-regular. Then
        $\units{R}$ and $\Jac(R)$ are closed.
    \end{lem}

    \begin{proof}
        Let $a\in \overline{\units{R}}$ and let $e$ be its associated idempotent. Then for any $I\in \calI_R$
        there is $a_I\in\units{R}$ such that $a-a_I\in I$. Clearly
        $e+I$ is the associated idempotent of $a+I=a_I+I$ in $R/I$. However, $a_I+I\in\units{(R/I)}$ and thus
        $1+I$ is its associated idempotent. It follows that $e+I=1+I$ for all $I\in\calI_R$, hence
        $e=1$ and $a\in\units{R}$.

        Now assume $a\in\overline{\Jac(R)}$. It is enough to show that for all $b\in R$,
        $1+ab\in \units{R}$. Let $I\in \calI_R$ and let $a_I\in \Jac(R)$ be
        such that $a-a_I\in I$. Then $1+a_Ib\in\units{R}$ and $(1+ab)-(1+a_Ib)\in I$.
        Therefore, $1+ab\in\bigcap_{I\in \calI_R}(\units{R}+I)=\overline{\units{R}}=\units{R}$.
    \end{proof}

    \begin{remark} \label{RING:RM:Jacobson-radical-might-be-open}
        (i) The assumption that $R$ is quasi-$\pi$-regular in the last lemma is essential;
        see Example \ref{RING:EX:non-hausdorff-I} (take $n=0$). In addition,
        $\Jac(R)^2$ need not be closed even when $R$ is quasi-$\pi$-regular; see Example \ref{RING:EX:JacRsquared-can-be-non-open}.

        (ii) If $R$ is quasi-$\pi$-regular and semiperfect, then $\units{R}$ and $\Jac(R)$
        are also open. Indeed, by the previous lemma $\Jac(R)=\overline{\Jac(R)}=\bigcap_{I\in\calI_R}(\Jac(R)+I)$,
        hence $\Jac(R)$ is an intersection of open ideals. Since $R/\Jac(R)$ is artinian, $\Jac(R)$ is
        the intersection of finitely many such ideals, thus open. The set $\units{R}$ is open
        since it is a union of cosets of $\Jac(R)$.
    \end{remark}

    In order to apply theorem \ref{RING:TH:main-res-for-quasi-pi-reg-rings} to pro-semiprimary
    rings, we need to
    recall some facts about complete topological rings. While the
    exact definition (see \cite[\S7-8]{Wa93}) is of little use to us, we will need
    the following results. Let $R\in\LTR$, then:
    \begin{enumerate}
        \item[(1)] $R$ is complete if and only if $R$ is isomorphic to an inverse limit of
        an inverse system of discrete
        topological rings $\{R_i\}_{i\in I}$. In this case, if $\vphi_i$ is the standard map from $R$ to $R_i$, then
        $\{\ker\vphi_i\where i\in I\}$ is a local basis of $R$.
        \item[(2)] If $R$ is complete
        and $\calB$ is a local basis consisting of ideals, then $R\cong \invlim \{R/I\}_{I\in\calB}$.
        (Note that $R/I$ is discrete for all $I\in \calB$.)
    \end{enumerate}
    We will also use the fact that a closed subring of complete ring is complete.
    (This can be verified directly for rings in $\LTR$ using the previous facts.)

    We now specialize the definition of pro-semiprimary rings given in Section \ref{section:preface}
    to topological rings. For a ring property $\calP$, a topological ring $R$ will be called \emph{pro-$\calP$}
    if $R$ is isomorphic as a topological ring to the inverse limit of an inverse system of discrete rings
    satisfying $\calP$. If in addition the standard map from $R$ to each of these rings is onto\footnote{
        This is not trivial since the maps in the inverse system are not assumed to be onto.
    }, then $R$ will
    be called \emph{strictly pro-$\calP$}. Clearly any pro-$\calP$ ring is complete and lies in $\LTR$.
    An LT ring $R$ is strictly pro-$\calP$
    if and only if it is complete and admits a local basis of ideals $\calB$ such that $R/I$ has $\calP$ for all $I\in\calB$.
    Notice that if $\calP$ is preserved under Morita equivalence, then so does being pro-$\calP$ and being strictly pro-$\calP$
    (because the isomorphisms of Lemma \ref{RING:LM:basic-inv-lim-lem} are also topological isomorphisms).

    \begin{remark}
        Any inverse limit of (non-topological) rings satisfying $\calP$ can
        be endowed with a linear ring topology making it into a pro-$\calP$ ring, but this topology
        usually depends on the inverse system used to construct the ring. However, when $\calP=\textrm{semiprimary}$
        and the ring is right noetherian, the topology is uniquely determined and always coincide with the Jacobson
        topology! See Section \ref{section:rings-with-Hausodroff-mods}.
    \end{remark}

    Once recalling Remark \ref{RING:RM:hirarchy-of-properties}, the following lemma implies
    that pro-semiprimary rings are quasi-$\pi$-regular.

    \begin{lem} \label{RING:LM:pro-pi-reg-lemma-I}
        Let $\{R_i,f_{ij}\}$ be an $I$-indexed inverse system of $\pi$-regular rings and let $R=\invlim \{R_i\}_{i\in I}$.
        Then $R$ is quasi-$\pi$-regular.
    \end{lem}

    \begin{proof}
        We identify $R$ with the set of compatible $I$-tuples in $\prod_{i\in I} R_i$
        (i.e.\ tuples $(x_i)_{i\in I}$ satisfying $f_{ij}(x_j)=x_i$ for all $i\leq j$ in $I$).
        Let $a=(a_i)_{i\in I}\in R$ and let $e_i\in\ids{R_i}$ be the associated idempotent of $a_i$ in $R_i$.
        The uniqueness of $e_i$ implies that $e=(e_i)_{i\in I}$ is compatible and
        hence lie in $R$.
        We claim that $e$ is the associated idempotent of $a$ in $R$. Conditions (A) and (C$'$)
        are straightforward, so
        we only check (B): Let $b_i$ be the inverse of $e_ia_ie_i$ in $e_iRe_i$.
        The for all $i\leq j$ in $I$, $f_{ij}(b_j)$ is also an inverse of $e_ia_ie_i$ in $e_iRe_i$,
        hence $f_{ij}(b_j)=b_i$. Therefore, $b:=(b_i)_{i\in I}$ is compatible and lie in $R$.
        Clearly $b=ebe$ and
        $b(eae)=(eae)b=e$ (since this holds in each coordinate), so condition (B) is satisfied.
        We thus conclude that $R$ is quasi-$\pi$-regular. \rem{To finish, if $\vphi_i$ is
        the natural map from $R\to R_i$, then $\vphi_i(R)\cong R/\ker\vphi_i$
        and the latter is $\pi$-regular since $\ker\vphi_i$ is open in $R$.}
    \end{proof}

    The converse of Lemma \ref{RING:LM:main-lemma-II-quasi-pi-reg} is almost true;
    if $R\in\LTR$ is quasi-$\pi$-regular, then $R$ is dense in a
    pro-$\pi$-regular ring, namely $\invlim \{R/I\}_{I\in\calI_R}$.
    The following theorem implies that T-semi-invariant subrings of semiperfect pro-semiprimary
    rings are semiperfect and pro-semiprimary (w.r.t.\ the induced topology).

    \begin{thm} \label{RING:TH:pro-semiprim-full-result}
        Assume $R=\invlim \{R_i\}_{i\in I}$ where each $R_i$ is $\pi$-regular.
        Denote by $J_i$ the kernel of the standard map $R\to R_i$ and let $R_0$ be a T-semi-invariant subring
        of $R$. Then:
        \begin{enumerate}
            \item[(i)] $R_0$ is quasi-$\pi$-regular and $R_0=\invlim \{R_0/(J_i\cap R_0)\}_{i\in I}$.
            \item[(ii)] If $R$ does not contain an infinite set of
            orthogonal idempotents, then $R_0$ is semiperfect and there is $n\in\N$ such that $\Jac(R_0)^n\subseteq \Jac(R)$.
            \item[(iii)] For all $i\in I$, $R_0/(J_i\cap R_0)$ is $\pi$-regular.
                If moreover $R_i$ is semiprimary (right perfect, semiperfect), then so is $R_0/(J_i\cap R_0)$. In particular,
                if $R$ is pro-semiprimary (pro-right-perfect, pro-$\pi$-regular-and-semiperfect),
                then so is $R_0$.
        \end{enumerate}
    \end{thm}

    \begin{proof}
        By Lemma \ref{RING:LM:pro-pi-reg-lemma-I}, $R$ is quasi-$\pi$-regular,
        so the first assertion of (i) is Corollary \ref{RING:CR:quasi-pi-reg-transfers-to-semi-inv-subring}(ii).
        As for the second assertion, $R$ is complete and $R_0$ is closed in $R$, hence $R_0$ is complete.
        Since $\{R_0\cap J_i\where i\in I\}$ is a local basis of $R_0$, $R_0=\invlim \{R_0/(J_i\cap R_0)\}_{i\in I}$.
        (ii) follows from Lemma \ref{RING:LM:main-lemma-II-quasi-pi-reg}(ii) and Theorem \ref{RING:TH:main-res-for-quasi-pi-reg-rings}.
        As for (iii), $R_0/(J_i\cap R_0)$ is $\pi$-regular as a quotient of a quasi-$\pi$-regular ring with an open ideal.
        The rest follows from Lemma \ref{RING:LM:main-lemma-III} (applied to $R_0/(J_i\cap R_0)$ identified as a subring of $R_i$).
    \end{proof}

    Let $\calP\in\{\textrm{semiprimary},~\textrm{right-perfect},~\textrm{$\pi$-regular-and-semiperfect},
    ~\textrm{$\pi$-regular}\}$. Then
    Theorem \ref{RING:TH:pro-semiprim-full-result} implies that pro-$\calP$
    rings
    are \emph{strictly} pro-$\calP$ (take $R_0=R$). In fact,
    we can  prove an even stronger result:

    \begin{cor} \label{RING:CR:inverse-lim}
        In the previous notation, the inverse limit of a small category of pro-$\calP$ rings
        is strictly pro-$\calP$.
    \end{cor}

    \begin{proof}
        Let $\CFont{C}$ be a small category of pro-$\calP$ rings and let $\{R_i\}_{i\in I}$
        be the objects of $\CFont{C}$.
        Then $R=\invlim \CFont{C}$ can be identified with the set of $I$-tuples
        $(x_i)_{i\in I}\in\prod_{i\in I}R_i$ such that $f(x_j)=x_i$
        for all $i,j\in I$ and $f\in \Hom_{\CFont{C}}(R_j,R_i)$.
        Clearly $S:=\prod_{i\in I}R_i$ is pro-$\calP$. If we can
        prove that $R$
        is a T-semivariant subring of $S$, then we are through by Theorem \ref{RING:TH:pro-semiprim-full-result}.
        Indeed, let $\pi_i$ denote the projection from $S$ to $R_i$.
        For all $i,j\in I$ and $f\in\Hom_{\CFont{C}}(R_j,R_i)$
        define $\vphi_f^{(1)},\vphi_f^{(2)}:S\to R_i$ by $\vphi_f^{(1)}=\pi_i$, $\vphi_f^{(2)}=f\circ\pi_j$.
        Then $R=\{x\in S\suchthat \vphi_f^{(1)}(x)=\vphi_f^{(2)}(x)\,\forall f\}$, hence $R$
        is a T-semi-invariant subring of $S$ by Proposition \ref{RING:PR:T-semi-invariant-definition-prop}(e).
    \end{proof}

    In some sense, Corollary \ref{RING:CR:inverse-lim} includes Theorem \ref{RING:TH:main-res-general}
    and part of
    Theorem \ref{RING:TH:pro-semiprim-full-result} because a T-semi-invariant subring
    can be understood as the inverse limit of a category with two objects. (Indeed, if $R\subseteq S\in\LTR$ and
    $\Sigma$ is a submonoid of $\cEnd(S)$, then take $\mathrm{Ob}(\CFont{C})=\{R,S\}$ with $\End_{\CFont{C}}(S)=\Sigma$,
    $\End_{\CFont{C}}(R)=\{\id_R\}$, $\Hom_{\CFont{C}}(S,R)=\phi$ and $\Hom_{\CFont{C}}(R,S)=\{i\}$ where
    $i:R\to S$ is the inclusion map).

\rem{
    In general, a semi-invariant subring of a semiperfect pro-semiprimary ring $R$ need not be
    semiperfect (see Example \ref{RING:EX:non-pro-semiprim-semi-inv-subring}). However,
    the author could not find an example with a \emph{closed} semi-invariant subring.
}

    \begin{cor} \label{RING:CR:pro-semiprim-is-satisfies-jacobson-conj}
            If $R$ is pro-semiprimary, then for any $J\in\calI_R$ there is $n\in\N$ such that $\Jac(R)^n\subseteq J$.
            In particular, $\bigcap_{n=1}^\infty \Jac(R)^n=\{0\}$.
    \end{cor}

    \begin{proof}
        Assume $R=\invlim \{R_i\}_{i\in I}$ with each $R_i$ semiprimary and let $J_i$ be as in Theorem
        \ref{RING:TH:pro-semiprim-full-result}. Since $\{J_i\where i\in I\}$ is local basis, there
        is $i\in I$ such that $J_i\subseteq J$. By Theorem
        \ref{RING:TH:pro-semiprim-full-result}(iii), $R/J_i$ is semiprimary, hence there is $n\in\N$
        such that $\Jac(R/J_i)^n=0$. As $\Jac(R/J_i)\supseteq(\Jac(R)+J_i)/J_i$, we get $\Jac(R)^n\subseteq J_i\subseteq J$.
    \end{proof}

    \begin{remark} \label{RING:RM:non-com-non-local-non-complete-quasi-pi-reg-ring}
        We will show in Proposition \ref{RING:PR:quasi-pi-reg-is-henselian} that Henselian rank-$1$ valuation
        rings are quasi-$\pi_\infty$-regular. In particular, non-complete such
        rings (e.g.\ the $\Q$-algebraic elements in $\Z_p$) are examples of non-complete quasi-$\pi_\infty$-regular rings.\rem{

        There are non-commutative
        quasi-$\pi$-regular rings which are not local and not complete. Let $1<n\in\N$ and let $S=\nMat{\Z_p}{n}$. Then $S$ is pro-semiprimary,
        hence quasi-$\pi$-regular. Define a sequence of subrings of $S$ as follows: Let $R_0=\nMat{\Z}{n}$ and
        given $R_i$, define $R_{i+1}$ to be the ring generated by $R_i$ and $\{e_r\where r\in R_i\}$
        where $e_r$ is the associated
        idempotent of $r$ in $S$. Finally, let $R=\bigcup_{i=0}^\infty R_i$. Then $R$ is
        quasi-$\pi$-regular (w.r.t.\ the topology induced from $S$). However, $R\neq S$ since
        $|R|=\aleph_0\neq |S|$. The ring $R$ is not complete nor commutative nor local since
        it contains $R_0$ which is dense in $S$ and contains matrix units.
        Note that remark \ref{RING:RM:remark-after-uniqueness-lemma}(iii) implies that in case $p=3$ and
        $n=4$, $R\neq \nMat{\Z_{\ideal{p}}}{n}$ (since the latter is not
        quasi-$\pi$-regular). The author believes this holds in general.
}
        In addition, the author suspects that the following are also explicit examples of non-complete quasi-$\pi_\infty$-regular
        rings:
        (1)
        The ring of power series $\sum_{i=0}^\infty a_it^i\in\Z_p\dbrac{t}$ with $a_i\to 0$
        endowed with
        the $t$-adic topology (such rings are common in rigid geometry).
        (2) The ring in the comment after Lemma \ref{RING:LM:basic-inv-lim-lem} w.r.t.\ its Jacobson topology.
    \end{remark}

\rem{
    We finish by noting that it is
    likely that the theory presented in this section can be extended to
    \emph{right} linearly topologized topologized rings, i.e.\ topological rings admitting a local basis
    consisting of right ideals.
}

\section{Counterexamples}
\label{section:examples}

    This section consists of counterexamples. In particular, we show that:
    \begin{enumerate}
        \item If $R$ is a semiperfect ring and $\Sigma\subseteq \End(R)$, then
            $R^\Sigma$ need not be semiperfect even when $\Sigma$ is a finite group and
            even when $\Sigma$ consists of a single automorphism.
            Similarly, if $X\subseteq R$ is a set, then
            $\Cent_R(X)$ need not be semiperfect even when $X$ consists of a single element.
        \item The semiperfect analogue of Corollary \ref{RING:CR:main-cr:I} is not true in general.
        \item A semi-invariant subring of a semiperfect pro-semiprimary ring need not be semiperfect even
        when closed
        (in contrast to T-semi-invariant subrings).
        \item Rationally closed subrings of a f.d.\ algebra need not be semiperfect. In particular,
            Theorem \ref{RING:TH:main-res-general} do not generalize to rationally closed
            subrings.
        \item No two of the families of semi-invariant, invariant, centralizer and rationally closed subrings
            coincide in general.
    \end{enumerate}
    We note that (1) is also true if we replace semiperfect with artinian. This
    was treated at the end of Section \ref{section:main-thm}.

    \smallskip

    We begin with demonstrating  (1).
    Our examples use Azumaya algebras and we refer the reader to \cite{Sa99} for definition and details.

    \begin{example} \label{RING:EX:inv-subring-semiperfect-case-I}
        Let $S$ be a discrete valuation ring with
        maximal ideal $\pi S$, residue field $k=S/\pi S$ and fraction field $F$,
        and let $A$ be an Azumaya algebra over $S$. Recall that this implies $A/\pi A$ is
        a central simple $k$-algebra and $\Jac(A)=\pi A$. Assume the following holds:
        \begin{enumerate}
            \item[(a)] $D=F\otimes_SA$ is a division ring.
            \item[(b)] $A/\pi A$ has zero divisors.
        \end{enumerate}
        In addition, assume
        there is a set $X\subseteq \units{A}$
        generating $A$ as an $S$-algebra (such $X$ always exists).
        Note that conditions (a) and (b) imply that $A$ is not semiperfect because $A$ contains
        no non-trivial idempotents while $A/\pi A=A/\Jac(A)$ does contain such idempotents,
        hence $\Jac(A)$ is not idempotent lifting.

        Define $R=A\otimes_S A^\op$ and let $\Sigma=\{\sigma_x\}_{x\in X}$ where
        $\sigma_x$ is conjugation by $1\otimes x^\op$. Then $R$ is an
        Azumaya $S$-algebra which is an $S$-order inside $D\otimes_F D^\op\cong \nMat{F}{r}$. This is well known to imply that $R
        \cong\End(P_S)$ for some faithful finite projective $S$-module $P$ (in fact, $P$ is free since $S$ is local). Therefore,
        $R$ is
        Morita equivalent to $S$, hence semiperfect.
        On the other hand,
        $R^\Sigma=
        \Cent_R(\{1\otimes x^\op \where x\in X\})=\Cent_R(S\otimes A^\op)=A\otimes S\cong A$,
        so $R^\Sigma$ is not semiperfect.

        An explicit choice for $S,A,F,D$ is $S=\Z_{\ideal{3}}$ ($\pi=3$), $F=\Q$,
        $D=(-1,-1)_{\Q}=\Q[\,i,j\,|\,ij=-ji, i^2=j^2=-1]$
        and $A=S[i,j]$. If we take $X=\{i,j\}$ then $\Sigma$ will
        consist of two inner automorphisms which are easily seen to generate
        an automorphism group isomorphic to $(\Z/2)\times (\Z/2)$.
    \end{example}

    \begin{example} \label{RING:EX:inv-subring-semiperfect-case-II}
        Let $S,\pi,A,F,D$ be as in Example \ref{RING:EX:inv-subring-semiperfect-case-I}
        and assume there is a cyclic Galois extension $K/F$ such that:
        \begin{enumerate}
            \item[(c)] $K/F$ is totally ramified at $\pi$.
            \item[(d)] $K\otimes_F D$ splits (i.e.\ $K\otimes_F D \cong \nMat{K}{t}$).
        \end{enumerate}
        Write $\nGal{K}{F}=\Trings{\sigma}$.

        Let $T$ denote the integral closure of $S$ in $K$. Then $\sigma(T)=T$.
        We claim that $T\otimes_S A$ is semiperfect, but $(T\otimes_S A)^{\{\sigma\otimes 1\}}$
        is not. Indeed, $T^{\{\sigma\}}=T\cap F=S$, so $(T\otimes A)^{\{\sigma\otimes 1\}}=S\otimes A\cong A$
        which is not semiperfect as explained in the previous example.
        On the other hand, $T\otimes A$ is an Azumaya $T$-algebra and a $T$-order in $K\otimes D\cong M_t(K)$.
        Again, this implies $T\otimes A$ is Morita equivalent to $T$. But
        $T$ is local because $K/F$ is totally ramified at $\pi$, hence $T\otimes A$
        is semiperfect.

        If we take $S,A,F,D$ as in the previous example, then $T=S[\sqrt{-3}]$, $K=\Q[\sqrt{-3}]$
        will satisfy (c) and (d).
        Indeed, dimension constraints imply $K\otimes D$ is either a division ring or $M_2(K)$,
        but
        $\sqrt{-3}+i+j+ij\in K\otimes D$ has reduced norm $0$, so the latter option must hold.\rem{\footnote{
            The Author thanks Dr.\ Eli Matzri for his help in choosing
            appropriate values for $S,A,F,D$ and $K$ in Examples \ref{RING:EX:inv-subring-semiperfect-case-I}
            and \ref{RING:EX:inv-subring-semiperfect-case-II}.
        }}
    \end{example}

    \begin{example} \label{RING:EX:inv-subring-semiperfect-case-III}
        Start with a semiperfect ring $R$ and $\sigma\in\End(R)$ such that $R^{\{\sigma\}}$ is not semiperfect
        (e.g.\ those of Example \ref{RING:EX:inv-subring-semiperfect-case-II}).
        Let $R'=R[[t;\sigma]]$
        be the ring of  $\sigma$-twisted formal power series
        with left coefficients in $R$
        (i.e.\ $\sigma(r)t=t r $ for all $r\in R$) and let $x=1+t\in\units{(R')}$. We claim $R'$ is semiperfect, but $\Cent_{R'}(x)$ is not.
        Indeed, $\Cent_{R'}(x)=\Cent_{R'}(t)=
        R^{\{\sigma\}}[[t]]$, so we are done by applying the following
        proposition for $R[[t;\sigma]]$ and $R^{\{\sigma\}}[[t]]$.
    \end{example}

    \begin{prp}
        For any ring $W$
        and $\tau\in\End(W)$, $W$ is
        semiperfect if and only if $W[[t;\tau]]$ is.
    \end{prp}

    \begin{proof}
        Let $V=W[[t;\tau]]$
        and let $J=\Jac(W)+Vt\idealof V$. Then $V/J\cong W/\Jac(W)$.
        Since
        the r.h.s.\ has zero Jacobson radical, $J\supseteq\Jac(V)$.
        However, as $1+J\subseteq\units{V}$, we have $J\subseteq\Jac(V)$, thus $\Jac(V)=J$.
        The isomorphism $V/J\cong W/\Jac(W)$ now implies that $V$ is semilocal $\iff$ $W$ is semilocal.
        We finish by observing that $Vt$ is idempotent lifting (this immediate as $V=W\oplus Vt$),
        hence $J$ is idempotent lifting in $V$ $\iff$ $J/Vt$ is idempotent lifting in $V/Vt$
        $\iff$ $\Jac(W)$ is idempotent lifting in $W$.
    \end{proof}

    We now show (2), relying on the previous examples.

    \begin{example} \label{RING:EX:main-cor-fails-in-general}
        Let $R$ be a semiperfect ring and let $X\subseteq R$ be such that $\Cent_R(X)$ is not
        semiperfect (the existence of such $R$ and $X$ was shown in previous examples).
        Let $Y=\{y_a~|~a\in X\}$ be a set of formal
        variables and let $S=R\Trings{Y}$ be the ring of non-commutative polynomials in $Y$ over $R$ ($Y$
        commutes with $R$). We can make $R$ into a right $S$-module by considering the standard right
        action of $R$ onto itself and extending it to $S$ by defining $r\cdot y_a = ar$ for all $a\in X$.
        Let $M$ denote the right $S$-module obtained thusly.
        Identify $R$ with $\End(M_R)=\End(R_R)$ via $r\mapsto(m\mapsto rm)\in\End(R_R)$.
        It is straightforward to check that $\End(M_S)$ now corresponds to  $\Cent_R(X)$.
        Therefore, $\End(M_R)\cong R$ is semiperfect but $\End(M_S)\cong \Cent_R(X)$ is not semiperfect.
    \end{example}

    The next example demonstrates (3).

    \begin{example} \label{RING:EX:non-pro-semiprim-semi-inv-subring}
        Let $p,q$ be distinct
        primes. Endow $R=\Z_p\times \Z_q\times \Q$ with the product topology
        (the topology on $\Q$ is the discrete topology). Then $R$
        is clearly semiperfect and pro-semiprimary. Define
        $K=\{(a,a,a)\where a\in\Q\}$ and let $R_0=R\cap K$. Then $R_0$ is a semi-invariant subring of $R$
        by Proposition \ref{RING:PR:semi-inv-examples-pr} (take $S=\Q_p\times \Q_q\times\Q$) and it
        is routine to check $R_0$ is closed. However, $R_0$ is not semiperfect. Indeed,
        it is isomorphic
        to  $T=M^{-1}\Z$ where $M=\Z\setminus (p\Z\cup q\Z)$.
        The ring $T$ is not semiperfect because it has no non-trivial idempotents
        while $T/\Jac(T)\cong T/pqT\cong T/pT\times T/qT$ has such, so $\Jac(T)$ cannot be idempotent lifting.
    \end{example}

    The following example shows that rationally closed subrings of a f.d.\ algebra
    need not be semiperfect. As f.d.\ algebras
    are semiprimary, this shows that Theorem \ref{RING:TH:main-res-general}
    fails for rationally closed subrings.

    \begin{example} \label{RING:EX:non-semi-inv-subring}
        Let $K=\Q(x)$ and $R=K\times K\times K$. Define $S'=\{f/g~|~f,g\in\Q[x],g(0)\neq 0,g(1)\neq 0\}$
        and observe that $S'$ is not semiperfect since it is
        a domain but $S'/\Jac(S')$ has non-trivial idempotents. (Indeed, $S'/\Jac(S')=S'/\ideal{x(x-1)}
        \cong S'/\ideal{x}\times S'/\ideal{x-1}\cong \Q\times\Q$). Define $\vphi:S'\to R$
        to be the $\Q$-algebra homomorphism obtained by sending $x$ to $a:=(0,1,x)\in R$ and let
        $S=\im \vphi$. It is easy to verify that $\vphi$ is well defined and injective, hence
        $S$ is not semiperfect.
        However, $S$ is rationally closed in $R$. To see this, let $q(x)\in S'$ and assume $q(a)\in\units{R}$.
        Then $q(0),q(1)\neq 0$. This implies $q(x)\in\units{(S')}$, hence $q(a)=\vphi(q(x))\in \units{S}$.
    \end{example}

    We finish by demonstrating (5).
    The subring $S$ of the last example
    cannot be semi-invariant, for otherwise we would get a contradiction to Theorem \ref{RING:TH:main-res-general}.
    In particular, $S$ is not an invariant subring nor a centralizer subring.
    Next, let $R=\Q[\sqrt[3]{2},\sqrt{3}]$. Then the only centralizer subring of $R$ is $R$ itself,
    the invariant subrings of $R$ are $R$ and $\Q[\sqrt[3]{2}]$ (Remark \ref{RING:RM:inv-subrings-of-a-field})
    and the semi-invariant subrings of $R$ are the four subfields of $R$ (Corollary \ref{RING:CR:semi-inv-subrings-of-a-field}). In particular,
    $R$ admits a semi-invariant non-invariant subring and an invariant non-centralizer subring.
    \rem{
    Corollary \ref{RING:CR:semi-inv-subrings-of-a-field} and Remark \ref{RING:RM:inv-subrings-of-a-field}
    suggest an easy construction of a semi-invariant subring that is not invariant.
    Finally, note that a commutative ring admits only one centralizer subring, which is the ring itself. Therefore, it
    is relatively easy to construct an invariant subring which is not a centralizer subring.
    }

\section{Applications}
\label{section:applications}

\newcommand{\Pdisc}{\CFont{P}_{\mathrm{disc}}}
\newcommand{\Ptop}{\CFont{P}_{\mathrm{top}}}
\newcommand{\Pmor}{\CFont{P}_{\mathrm{mor}}}
\newcommand{\Psp}{\CFont{P}_{\mathrm{sp}}}

    This section presents applications of the previous results.
    In order to avoid cumbersome phrasing, we introduce the following
    families of ring-theoretic properties:
    \begin{eqnarray}
    \Pdisc  &\!\!\!\!= & \!\!\!\!
    \left\{\begin{array}{c}
    \textrm{semiprimary},~\textrm{right perfect},~\textrm{left perfect},~
    \textrm{$\pi_\infty$-regular and semiperfect}, \\
    \textrm{$\pi_\infty$-regular},~\textrm{$\pi$-regular and semiperfect},~\textrm{$\pi$-regular}
    \end{array}\right\}
    \nonumber\\
    \Ptop &\!\!\!\!= & \!\!\!\!
    \left\{\begin{array}{c}
    \textrm{pro-$\calP$},~\textrm{pro-$\calP$ and semiperfect}, \\
    \textrm{quasi-$\calQ$},~\textrm{quasi-$\calQ$ and semiperfect}
    \end{array}\Bigg|~
    \begin{array}{c}
    \calP\in \Pdisc \\
    \calQ\in\{\textrm{$\pi_\infty$-regular},~\textrm{$\pi$-regular}\}
    \end{array} \right\}
    \nonumber
    \end{eqnarray}
    (For example, ``quasi-$\pi_\infty$-regular and semiperfect'' lies in $\Ptop$.)
    Note that the properties in $\Pdisc$ apply to rings while the properties in $\Ptop$
    apply to LT rings. Nevertheless, we will sometimes address non-topological
    rings as satisfying one of the properties of $\Ptop$, meaning that they satisfy it w.r.t.\ \emph{some} linear ring
    topology. We also define $\Pmor$ (resp.\ $\Psp$) to be the set of properties in $\Pdisc\cup\Ptop$ which
    are preserved under Morita equivalence (resp.\ imply that the ring is semiperfect).\rem{
    \begin{eqnarray}
    \Pmor &\!\!\!\!= & \!\!\!\!
    \{\calP\in\Pdisc\cup\Ptop\suchthat \textrm{$\calP$ is preserved under Morita equivalence}\}
    \nonumber\\
    &\!\!\!\!=& \!\!\!\!\{\textrm{the properties in $\Pdisc\cup\Ptop$ not containing the string ``$\pi$-regular''}\}
    \nonumber\\
    \Psp &\!\!\!\!= & \!\!\!\!
    \{\calP\in\Pdisc\cup\Ptop\suchthat \textrm{$\calP$ implies semiperfect}\}
    \nonumber
    \end{eqnarray}}
    Recall that a  property  in $\Ptop$ is
    preserved under Morita equivalence if this holds in the sense of Section \ref{section:top-rings}
    (and not in the sense of \cite{Gr88}). For example, ``$\pi$-regular'' and ``quasi-$\pi$-regular''
    do not lie in $\Pmor$ nor in $\Psp$, ``pro-semiprimary and semiperfect'' lies in both $\Psp$ and $\Pmor$,
    and
    ``pro-semiprimary'' lies in $\Pmor$, but not in $\Psp$.

    Theorems \ref{RING:TH:main-res-general}, \ref{RING:TH:main-res-for-quasi-pi-reg-rings} and \ref{RING:TH:pro-semiprim-full-result}
    and Corollaries \ref{RING:CR:pi-reg-transfers-to-semi-inv-subring}(ii) and \ref{RING:CR:quasi-pi-reg-transfers-to-semi-inv-subring}(ii)
    can be now summarized as follows:

    \begin{thm} \label{RING:TH:summary-thm}
        Let $R$ be a ring and $R_0$ a subring.
        \begin{enumerate}
        \item[(i)] If $R$ has $\calP\in\Pdisc$ and $R_0$ is semi-invariant, then $R_0$ has $\calP$.
        \item[(ii)] If $R\in\LTR$ has $\calP\in\Ptop$ and $R_0$ is T-semi-invariant, then $R_0$ has $\calP$
        (w.r.t.\ the induced topology).
        \item[(iii)] In both (i) and (ii), if $\calP\in \Psp$, then $\Jac(R_0)^n\subseteq\Jac(R)$ for some $n\in\N$.
        \end{enumerate}
    \end{thm}

    Our first application follows from the fact that a centralizer subring is always \mbox{(T-)semi-invariant}:

    \begin{cor}
        Let $\calP\in \Pdisc\cup\Ptop$ and let $R$ be a ring satisfying $\calP$. Then
        $\Cent(R)$ and any
        maximal commutative subring of $R$ satisfy $\calP$.
    \end{cor}

    \begin{proof}
        $\Cent(R)$ is the centralizer of $R$ and a maximal commutative subring of $R$ is itself's centralizer.
        Now apply Theorem \ref{RING:TH:summary-thm}.
    \end{proof}

    Surprisingly, the author could not find in the literature results that are similar to the previous corollary,
    except the fact that the center of a right artinian ring is semiprimary. (This follows from a classical result of
    Jacobson, stating that the endomorphism ring of any module of finite length is semiprimary, together with the
    fact that the center of a ring $R$ is isomorphic to $\End({}_RR_R)$.)

    \smallskip

    The next applications concern endomorphism rings of finitely presented modules.
    We will only treat here the non-topological properties (i.e.\ $\Pdisc$). The
    topological analogues of the results to follow require additional notation
    and are thus postponed to the next section.

    \begin{thm} \label{RING:TH:application-to-fp-modules}
        Let $R$ be a ring
        satisfying $\calP\in \Pdisc\cap\Pmor$
        and let $M$ be a finitely presented right $R$-module. Then $\End(M_R)$ satisfies $\calP$.
    \end{thm}

    \begin{proof}
        There is an exact sequence $R^n\to R^m\to M\to 0$ with $n,m\in\N$. Since $R^n, R^m$ are projective,
        we may apply Proposition \ref{RING:PR:exact-seq-lemma} to deduce that $\End(M)$ is a quotient
        of a semi-invariant subring of $\End(R^n)\times \End(R^m)\cong\nMat{R}{n}\times\nMat{R}{m}$. The latter
        has $\calP$ because any $\calP\in\Pdisc\cap\Pmor$
        is preserved under
        Morita equivalence and under taking finite products. Since all ring properties
        in $\Pdisc$ pass to quotients, we are done by Theorem \ref{RING:TH:summary-thm}.
    \end{proof}

    \begin{cor} \label{RING:CR:application-to-fp-modules}
        Let $\vphi:S\to R$ be a ring homomorphism. Consider $R$ as a right $S$-module via $\vphi$
        and assume it is finitely presented. Then if $S$ satisfies $\calP\in\Pdisc\cap\Pmor$, then so
        does $R$.
    \end{cor}

    \begin{proof}
        By Theorem \ref{RING:TH:application-to-fp-modules}, $\End(R_S)$ has $\calP$. Therefore,
        by Corollary \ref{RING:CR:main-cr:I}, $R\cong\End(R_R)$
        has $\calP$.
    \end{proof}

    \begin{remark}
        Theorem \ref{RING:TH:application-to-fp-modules} actually follows from results of Bjork, who proved the semiprimary case and part
        of the left/right perfect cases (\cite[Ths.\ 4.1-4.2]{Bj71B}), and Rowen, who proved the left/right perfect
        and the semiperfect and $\pi_\infty$-regular cases (\cite[Cr.\ 11 and Th.\ 8(iii)]{Ro86}).
        Our approach suggests a single simplified proof to all the cases.
        Note that we cannot replace ``finitely presented'' with ``finitely generated'' in Theorem \ref{RING:TH:application-to-fp-modules};
        in \cite[Ex.\ 2.1]{Bj71A}, Bjork presents a right artinian ring with a cyclic left module having
        a non-semilocal endomorphism ring. See also \cite[Ex.\ 3.5]{FacHer06}.
    \end{remark}

    By arguing as in the proof of Theorem \ref{RING:TH:application-to-fp-modules} one can also obtain:

    \begin{thm} \label{RING:TH:application-to-exact-sequences}
        Let $0\to A\to B\to C\to 0$ be an exact sequence in an abelian category $\CFont{A}$ such that $B$ is projective.
        \begin{enumerate}
            \item[(i)] If
            $\End(A)$ and $\End(B)$ has $\calP\in \Pdisc$, then
            $\End(C)$ has $\calP$.
            \item[(ii)] If $\End(A\oplus B)$ has $\calP\in \Psp$, then $\End(C)$ is semiperfect.
        \end{enumerate}
    \end{thm}

    Next, we turn to representations over modules with ``good'' endomorphism ring. By a representation
    of a monoid (ring) $G$ over a right $R$-module $M$, we mean a mononid (ring) homomorphism $\rho:G\to\End(M)$ (so
    $G$ acts on $M$ via $\rho$).

    \begin{cor}
        Let $R$ be a ring and let $\rho$ be a representation of a monoid (or a ring) $G$ over a right $R$-module $M$. Assume that one
        of the following holds
        \begin{enumerate}
            \item[(i)] $\End(M)$ has $\calP\in\Pdisc\cup\Ptop$.
            \item[(ii)] There is a sub-monoid (subring) $H\subseteq G$ such that $\End(\rho|_H)$ has
            $\calP\in \Pdisc$.
            \item[(iii)] $\End(M)$ is LT and Hausdorff and there is a sub-monoid (or a subring) $H\subseteq G$ such that $\End(\rho|_H)$
            has $\calP\in\Ptop$ w.r.t.\ the induced topology.
        \end{enumerate}
        Then $\End(\rho)$ has $\calP$. Moreover, if $\calP\in\Psp$, then $\rho$ has
        a Krull-Schmidt decomposition $\rho\cong\rho_1\oplus\dots\oplus \rho_t$
        and $\End(\rho_i)$ is local and has $\calP$ for all $1\leq i\leq t$.
    \end{cor}

    \begin{proof}
        (i) follows from (ii) and (iii) if we take $H$ to be the trivial monoid (basic subring). To see (ii) (resp.\ (iii)),
        notice that $\End(\rho)=\Cent_{\End(\rho|_H)}(\rho(G))$. Therefore,
        $\End(\rho)$ is a semi-invariant (resp.\ T-semi-invariant) subring of $\End(\rho|_H)$, hence
        by Theorem \ref{RING:TH:summary-thm}, $\End(\rho)$ has $\calP$.
        If $\calP\in\Psp$, then $\End(\rho)$ is semiperfect. The Krull-Schmidt
        Theorem now implies $\rho$ has a Krull-Schmidt decomposition
        $\rho\cong\rho_1\oplus\dots\oplus \rho_t$ and $\End(\rho_i)$ is local for all $i$. Finally,
        $\End(\rho_i)\cong e\End(\rho)e$
        for some $e\in\ids{\End(\rho)}$. Since for any ring $R$ and $e\in\ids{R}$, $R$ has $\calP$
        implies $eRe$ has $\calP$, we are through.
    \end{proof}

    Assume $R$ is a ring and $M$ is a right $R$-module such that $\End(M_R)$ is semiperfect and quasi-$\pi$-regular
    (see Theorem \ref{RING:TH:application-to-fp-modules-top-case} below for
    cases when this happens).
    Then the endomorphisms of $M$ have a ``Jordan decomposition'' in the following
    sense: If $f\in\End(M_R)$, then we can consider $M$ as a right $R[x]$-module by letting $x$ act as $f$.
    Clearly $\End(M_{R[x]})=\Cent_{\End(M_R)}(f)$, so by Theorem \ref{RING:TH:main-res-for-quasi-pi-reg-rings},
    $\End(M_{R[x]})$ is semiperfect. Therefore, $M_{R[x]}$ has a Krull-Schmidt decomposition $M=M_1\oplus\dots\oplus M_t$.
    (Notice that each $M_i$ is an $f$-invariant submodule of $M$.) This decomposition plays the role of a Jordan decomposition
    for $f$, since the isomorphism classes of $M_1,\dots,M_t$ (as $R[x]$-modules) determine the conjugation class of $f$.
    In particular, studying endomorphisms of $M$ can be done by classifying LE-modules over $R[x]$.

    \smallskip

    Finally, the results of this paper can be applied in a rather different manner
    to bilinear forms: Let $*$ be an anti-endomorphism of a ring $R$ (i.e.\
    an additive, unity-preserving map that reverses order of multiplication). Then $\sigma=*^2$
    is an endomorphism of $R$ and $*$ becomes an involution on the invariant subring $R^{\{\sigma\}}$.
    As some claims on $(R,*)$ can be reduced to claims on $(R^{\{\sigma\}},*|_{R^{\{\sigma\}}})$,
    our results become a useful tool for studying the former. Recalling that bilinear (resp.\ sesquilinear) forms correspond certain anti-endomorphisms and
    quadratic (resp.\ hermitian) forms correspond to involutions (see \cite[Ch.\ I]{InvBook}), these ideas, taken much further,
    can be used to reduce the isomorphism problem of bilinear forms to the isomorphism problem of hermitian forms.
    This was actually done (using other methods) for bilinear forms over fields by Riehm (\cite{Ri74}),
    who later generalized this with Shrader-Frechette to sesquilinear forms over semisimple algebras (\cite{RiSh76}).
    The author can improve these results for bilinear (sesquilinear) forms over various semiperfect pro-semiprimary rings (e.g.\
    f.g.\ algebras over $\Z_p$). This\rem{ approach is described at \cite{Fi12B} and} will be published elsewhere.

\section{Modules Over Linearly Topologized Rings}
\label{section:modules}

    In this section we extend Theorem \ref{RING:TH:application-to-fp-modules}
    and other applications to LT rings.
    This is done by properly topologizing modules and endomorphisms rings of modules
    over LT rings.

    \smallskip

    Let
    $R$ be an LT ring and let $M$ be a right $R$-module. Then $M$ can be made into a topological $R$-module
    by taking $\{x+MJ\where J\in \calI_R\}$ as basis of neighborhoods of $x\in M$. (That $M$ is indeed
    a topological module follows from \cite[Th.\ 3.6]{Wa93}.) Notice that any homomorphism
    of modules is continuous w.r.t.\ this topology. Furthermore, $\End(M)$ can be linearly topologized by
    taking $\{\Hom_R(M,MJ)\where J\in \calI_R\}$ as a local basis.\footnote{
        This topology is the uniform convergence topology (w.r.t.\ the natural uniform structure of $M$).
        If $M$ is f.g.\ then this topology coincides with the pointwise convergence topology (i.e.\
        the topology induced from the
        the product topology on $M^M$).
    } We will refer to the topologies just defined on $M$ and $\End(M)$ as
    their \emph{standard topologies}. In general, that $R$ is
    Hausdorff does not imply  $M$ or $\End(M)$ are Hausdorff. (E.g.: For
    any distinct primes $p,q\in\Z$, the $\Z$-module $\Z/q$ is not Hausdorff
    w.r.t.\
    the $p$-adic topology on $\Z$). Observe that $\overline{\{0_{\End(M)}\}}=\bigcap_{J\in\calI_R}\Hom(M,MJ)=
    \Hom(M,\bigcap_{J\in\calI_R}MJ)=\Hom(M,\overline{\{0_M\}})$, so $M$ is Hausdorff implies $\End(M)$ is Hausdorff.

    Now let $E_\bullet$ be a be a finite resolution of $M$, i.e.\ $E_\bullet$
    consists of an exact sequence $E_{n-1}\to \dots\to E_0\to E_{-1}=M\to 0$.\footnote{
        We do not require the map $E_{n-1}\to E_{n-2}$ to be injective.
    }
    The maps $E_i\to E_{i-1}$ will be denoted by $d_i$.
    We say that $E_\bullet$ has the \emph{lifting property} if any $f_{-1}\in\End(M)$
    can be extended to a chain complex homomorphism $f_\bullet:E_\bullet\to E_\bullet$. (Recall
    that $f_\bullet$
    consists of a sequence $\{f_i\}_{i=-1}^{n-1}$ such that $f_i\in\End(E_i)$ and $d_if_i=f_{i-1}d_i$ for all $i$.)
    In other words, $E_\bullet$ has the lifting property if and only if the following commutative diagram can
    be completed for every $f_{-1}\in \End(M)$.
    \[\xymatrix{
        E_{n-1} \ar[r]\ar@{.>}[d]^{f_{n-1}} & \dots \ar[r] & E_1 \ar[r]\ar@{.>}[d]^{f_1} & E_0 \ar[r]\ar@{.>}[d]^{f_0} & M \ar[d]^{f_{-1}}\\
        E_{n-1} \ar[r]                & \dots \ar[r] & E_1 \ar[r]             & E_0 \ar[r]            & M
    }\]
    For example, any projective resolution has the lifting property.
    We define a linear ring topology $\tau_E$ on $\End(M)$ as follows: For all
    $J\in\calI_R$, define $\Ball(J,E)$ to the set of maps
    $f_{-1}\in\End(M,MJ)$ that extend to a chain complex homomorphism $f_\bullet:E_\bullet\to E_\bullet$
    such that $\im f_i\subseteq E_iJ$ for all $-1\leq i<n$. The lifting property implies $\Ball(J,E)\idealof \End(M)$
    and it is clear that $\calB_E:=\{\Ball(J,E)\where J\in\calI_R\}$ is a filter base.
    Therefore, there is a unique ring topology  on $\End(M)$, denoted $\tau_E$, having $\calB_E$ as a local basis.

    It turns out that
    if $E_\bullet$ is a \emph{projective} resolution, then $\tau_E$ only
    depends on the length of $E$, i.e.\ the number $n$. Indeed, if $P_\bullet, P'_\bullet$
    are two \emph{projective} resolutions of length $n$ of $M$, then the map $\id_M :M\to M$
    gives rise to chain complex homomorphisms $\alpha_{\bullet}:P_\bullet\to P'_\bullet$
    and $\beta_{\bullet}:P'_\bullet\to P_\bullet$ with $\alpha_{-1}=\beta_{-1}=\id_M$.
    Now, if $J\in\calI_R$ and $f_{-1}\in\Ball(J,P)$, then there is $f_\bullet:P_\bullet\to P_\bullet$
    such that $\im f_i\subseteq P_iJ$ for all $i$. Define $f'_\bullet=\alpha_\bullet f_\bullet \beta_\bullet$.
    Then $\im f'_i\subseteq \alpha_i(P_iJ)\subseteq P'_iJ$ for all $i$ and $f'_{-1}=\id_M f_{-1}\id_M=f_{-1}$,
    so $f_{-1}\in\Ball(J,P')$. By symmetry, we get $\Ball(J,P)=\Ball(J,P')$ for all $J\in\calI_R$, hence
    $\tau_P=\tau_{P'}$.

    The topology of $\End(M)$ obtained from a projective resolution of length $n$ will
    be denoted by $\tau_n^M$ and the closure of the zero ideal in that topology will be denoted by $I_n^M$.
    Note that $\tau_1^M\subseteq \tau_2^M\subseteq\dots$ and that
    that $\tau_1^M$ is the standard topology on $\End(M)$. (Indeed, if $P_\bullet:~P_0\to M\to 0$ is a projective
    resolution of length $1$, then any $f\in\Hom(M,MJ)$ can be lifted to $f_0:P_0\to P_0J$ because the map
    $P_0J\to MJ$ is onto, hence $\Ball(P,J)=\Hom(M,MJ)$.) More generally,
    for any resolution $E_\bullet$ of $M$, $\tau_E$ contains the standard topology on $\End(M)$.
    Therefore, if $M$ is Hausdorff, then
    $\tau_E$ is Hausdorff.
    In the appendix we provide sufficient conditions for
    $\tau_1^M,\tau_2^M,\dots$ to coincide.

    With this terminology at hand, we can generalize Proposition \ref{RING:PR:exact-seq-lemma}:

\rem{
    \begin{prp}
        Let $R$ be an LT ring, let $S$ be a (non-topological) ring containing $R$ and let $M$ be
        a right $S$-module. Assume $M_R$ is Hausdorff and endow $\End(M_R)$ with its natural topology. Then
        $\End(M_S)$ is a T-semi-invariant subring of $\End(M_R)$.
    \end{prp}

    \begin{proof}
        Let $C^0(M)$ be the ring of all continuous functions from
        $M$ to itself and endow it with the uniform convergence topology.
        Then $C^0(M)$ is a Hausdorff \emph{right} linearly topologized ring
        with local basis $\{\mathrm{Cont}(M,MJ)\where J\in\calI_R\}$.
        Now
        continue as in Proposition \ref{RING:PR:endo-ring-is-semi-invariant}.
    \end{proof}
}

    \begin{prp} \label{RING:PR:exact-sequence-LT-rings}
        Let $R$ be an LT ring and let $E:~A\to B\to C\to 0$ be an exact sequence of right $R$-modules satisfying
        the lifting property (w.r.t.\ $C$)
        and such that $A$ and $B$ are Hausdorff.
        Assign $\End(A)$ and $\End(B)$ the natural topology and endow $\End(C)$ with $\tau_E$.
        Then $\End(C)$ is isomorphic as a topological ring to a quotient of
        a T-semi-invariant subring of $\End(A)\times \End(B)$.
    \end{prp}

    \begin{proof}
        We use the notation of the proof of Proposition
        \ref{RING:PR:exact-seq-lemma}. By that proof, $\End(C)$ is isomorphic to a quotient of
        $\Cent_{D}(\smallSMatII{0}{f}{0}{0})$.
        It is easy to check that the embedding
        $D\hookrightarrow S$
        is a topological embedding, hence $\End(C)$ is isomorphic to a quotient of
        a T-semi-invariant subring of $D$.
        That the quotient topology on $\End(C)$ is indeed $\tau_E$ is routine.
    \end{proof}

    We are now in position to generalize previous results.

    \begin{lem} \label{RING:LM:pro-P-equiv-cond}
        Let $R\in\LTR$ and $\calP\in\Pdisc$. Then $R$ is pro-$\calP$
        if any only if $R$ is complete and $R/I$ has $\calP$ for all $I\in \calI_R$.
    \end{lem}

    \begin{proof}
        If $R$ is complete and $R/I$ has $\calP$ for all $I\in \calI_R$, then $R\cong\invlim \{R/I\}_{I\in\calI_R}$,
        so $R$ is pro-$\calP$. On the other hand, if $R$ is pro-$\calP$, then it is complete.
        In addition, it is strictly pro-$\calP$
        (Corollary \ref{RING:CR:inverse-lim}), hence there is a local basis of ideals $\calB$
        such that $R/I$ has $\calP$ for all $I\in\calB$. Now, let $I\in\calI_R$. Then
        there is $I_0\in\calB$ contained in $I$. Now, $R/I$ is a quotient of $R/I_0$. As the latter has
        $\calP$, so does $R/I$.
    \end{proof}

    Recall that a topological ring is \emph{first countable} if it admits a countable local basis. If $R$ is pro-$\calP$,
    then this is equivalent to saying that $R$ is the inverse
    limit of countably many discrete rings satisfying $\calP$.

    \begin{thm} \label{RING:TH:application-to-fp-modules-top-case}
        Let $R\in\LTR$ be a ring and let $M$
        be a f.p.\ right $R$-module.
        \begin{enumerate}
            \item[(i)] If $R$ is first countable and satisfies $\calP\in\Ptop\cap\Pmor$, then $\End(M)/I_2^M$ satisfies $\calP$
            when $\End(M)$ is endowed with $\tau_2^M$.
            In particular, if $M$ is Hausdorff, then $\End(M)$ has $\calP$.
            \item[(ii)] Assume
            $R$ is quasi-$\pi_\infty$-regular and let $i\in\{1,2\}$.
            Then $\End(M)/I_i^M$ is quasi-$\pi_\infty$-regular when $\End(M)$ is endowed with $\tau_i^M$.
            In particular, if $M$ is Hausdorff, then $\End(M)$ is quasi-$\pi_\infty$-regular w.r.t.\ $\tau_1^M$.
            \item[(iii)] If $R$ is semiperfect and quasi-$\pi_\infty$-regular, then $\End(M)$ is semiperfect.
        \end{enumerate}
    \end{thm}

    \begin{proof}
        (i) The argument in the proof of Theorem \ref{RING:TH:application-to-fp-modules}
        shows that $\End(M)$ is a quotient of an LT Hausdorff ring satisfying $\calP$, which we denote by $W$
        (use Proposition \ref{RING:PR:exact-sequence-LT-rings} instead of Proposition \ref{RING:PR:exact-seq-lemma}).
        $I_2^M$ is a closed ideal of $\End(M)$ and therefore $\End(M)/I_2^M$ is a
        quotient of $W$ by a closed ideal. We finish by claiming that for any closed ideal $I\idealof W$,
        $W/I$ satisfies
        $\calP$. We will only check the case $\calP=\textrm{pro-$\calQ$}$
        for $\calQ\in \Pdisc$. The other cases are straightforward or follow from the pro-$\calQ$ case.
        Indeed, any open ideal of $W/I$ is of the form $J/I$ for some $J\in\calI_W$,
        hence by Lemma \ref{RING:LM:pro-P-equiv-cond}, $(W/I)/(J/I)\cong W/J$ satisfies $\calQ$.
        In addition, that $R$ is first countable implies $W$
        is first countable, hence by the Birkhoff-Kakutani Theorem,
        $W$ is metrizable. By \cite[p.\ 163]{Bo66} a Hausdorff quotient
        of a complete metric ring is complete, hence $W/I$ is complete.
        Therefore, by Lemma \ref{RING:LM:pro-P-equiv-cond} (applied to $W/I$),
        $W/I$ is pro-$\calQ$.

        (ii) The case $i=2$ follows from the argument of (i) since being $\pi$-regular
        passes to quotients by closed ideals (the first countable assumption is not needed). As for $i=1$,
        Since $I_1^M$ is closed in $\tau_1^M$, it is also closed in $\tau_2^M$.
        Therefore, $\End(M)/I_1^M$ is quasi-$\pi_\infty$-regular
        when $M$ is equipped with $\tau_2^M$.
        We are done by observing that if a ring is quasi-$\pi_\infty$-regular
        w.r.t.\ a given topology, then it is quasi-$\pi_\infty$-regular w.r.t.\ any linear Hausdorff sub-topology.

        (iii) By (i), $\End(M)$ is a quotient of a semiperfect ring, namely $W$.
    \end{proof}

    \begin{remark}\label{RING:RM:results-related-to-KSD}
        Part (iii) of Theorem \ref{RING:TH:application-to-fp-modules-top-case} was proved in \cite{Ro87}
        for complete semilocal rings with Jacobson radical f.g.\ as a right ideal and in \cite{Ro86}
        for semiperfect $\pi_\infty$-regular rings. Both conditions are included in being semiperfect and quasi-$\pi_\infty$-regular.
        In addition, V{\'a}mos proved in \cite[Lms.\ 13-14]{Vamos90} that all \emph{finitely generated}
        or torsion-free of finite rank modules rank over a  Henselian integral domain\footnote{
            A commutative ring $R$ is called \emph{Henselian} if $R$ is local and \emph{Hensel's Lemma} applies to $R$.
        } have semiperfect endomorphism ring.\rem{ Moreover, V{\'a}mos showed that the property that
        all f.g.\ modules have a Krull-Schmidt Decomposition characterizes the Henselian rings
        inside certain families of domains (e.g.\ Dedekind domain). fact that all}
        Results of similar flavor were
        also obtained in \cite{FacHer06}, where it is shown that the endomorphism ring
        of a f.p.\ (resp.\ f.g.) module over a semilocal (resp.\ commutative semilocal) ring is semilocal.
    \end{remark}

    \begin{cor} \label{RING:CR:fp-algebras-full-cor-topological-version}
        Let $S$ be a commutative LT ring and let $R$ be an $S$-algebra s.t. $R$ is f.p.\ and Hausdorff as
        an $S$-module. Then:
        \begin{enumerate}
            \item[(i)] If $S$ is quasi-$\pi_\infty$-regular,
            then $R$ is quasi-$\pi_\infty$-regular (w.r.t. to \emph{some} linear ring topology).
            If moreover $S$ is semiperfect, then so is $R$.
            \item[(ii)] If $S$ satisfies $\calP\in\Ptop\cap \Pmor$
            w.r.t.\ a given topology which is also first countable, then $R$ satisfies $\calP$.
        \end{enumerate}
    \end{cor}

    \begin{proof}
        We only prove (ii); (i) is similar. By Theorem \ref{RING:TH:application-to-fp-modules}, $\End(R_S)$ satisfies $\calP$.
        For all $r\in R$, define $\what{r}\in\End(R_S)$ by $\what{r}(x)=xr$ and observe that $\Cent_{\End(R_S)}(\{\what{r}\where r\in R\})
        \cong\End(R_R)=R$, hence $R$ has $\calP$ by Theorem \ref{RING:TH:summary-thm}.
    \end{proof}

    Let $C$ be a commutative local ring.
    Azumaya proved in \cite[Th.\ 22]{Azu51} that $C$ is Henselian if and only if every
    \emph{commutative} $C$-algebra $R$ with $R_C$ f.g.\ is semiperfect.
    This was improved by V{\'a}mos to non-commutative $C$-algebras in which all non-units are integral over $C$;
    see \cite[Lm.\ 12]{Vamos90}.
    Given the previous corollary, Azumaya and V{\'a}mos' results suggest that the notions of
    Henselian and quasi-$\pi_\infty$-regular might sometimes coincide. This is verified in the following proposition.

    \begin{prp}\label{RING:PR:quasi-pi-reg-is-henselian}
        Let $R$ be a rank-$1$ valuation ring. Then $R$ is Henselian if and only if
        $R$ is quasi-$\pi_\infty$-regular w.r.t.\ the topology induced by the valuation.
    \end{prp}

    \begin{proof}
        Assume $R$ is quasi-$\pi_\infty$-regular. Observe that any free $R$-module is Hausdorff w.r.t.\
        the standard topology, hence Corollary \ref{RING:CR:fp-algebras-full-cor-topological-version}(i)
        implies that any $R$-algebra $A$ such that $A_R$ is free of finite rank is semiperfect. Thus, by
        \cite[Th.\ 19]{Azu51}, $R$ is Henselian.

        Conversely, assume $R$ is Henselian. Denote by $\nu$
        the (additive) valuation of $R$. Since $\nu$ is of rank $1$, we may assume $\nu$
        take values in $(\R,+)$. For every $\delta\in \R$, let $I_\delta=\{x\in R\where \nu(x)>\delta\}$.
        Then $\{\nMat{I_\delta}{n}\where \delta\in [0,\infty)\}$ is a local basis for $\nMat{R}{n}$.
        Let $a\in\nMat{R}{n}$. By
        the Cayley-Hamilton theorem, $a$ is integral over $R$, hence $R[a]$
        is a f.g.\ $R$-module. Let $J=\Jac(R)\cdot R[a]$.
        Then $J\idealof R[a]$ and it is well known that $J\subseteq\Jac(R[a])$.
        The ring $R[a]/J$ is artinian, hence $a+J$ has an associated idemptent
        $\veps\in\ids{R[a]/J}$ (i.e.\ $\veps$ satisfies conditions (A)--(C) of Lemma \ref{RING:PR:pi-reg-equiv-cond}).
        By \cite[Th.\ 22]{Azu51}, $J$ is idempotent lifting, hence there is $e\in\ids{R[a]}$
        such that $e+J=\veps$. Let $f=1-e$.
        Then $a=eae+faf$ (since $R[a]$ is commutative). Furthermore,
        $eae+J$ is invertible in $\veps(R[a]/J)\veps$, hence $eae$ is invertible
        in $eR[a]e$ and in particular in $e\nMat{R}{n}e$. Next,  $(faf)^k\in J\subseteq\nMat{\Jac(R)}{n}=\nMat{I_0}{n}$
        for some $k\in\N$. This means $(faf)^k\in\nMat{I_\delta}{n}$ for some $0<\delta\in\R$,
        which implies $(faf)^m\xrightarrow{m\to\infty}0$.
        Thus, $e$ satisfies conditions (A),(B) and (C$'$) w.r.t.\ $a$ and we may conclude that $\nMat{R}{n}$
        is quasi-$\pi$-regular for all $n\in\N$.
    \end{proof}

    Using the ideas in the proof of Theorem \ref{RING:TH:application-to-fp-modules-top-case}, we can also obtain:

    \begin{thm}
        Let $R$ be an LT ring and let $E:~0\to A\to B\to C\to 0$ be an exact sequence
        of right $R$-modules such that $B$ is projective and $A$ and $B$ are Hasudorff.
        Endow $\End(A)$ and $\End(B)$ with their standard topologies and $\End(C)$ with $\tau_E$ and let
        $I_E$ denote the closure of the zero ideal in $\End(C)$.
        Then:
        \begin{enumerate}
            \item[(i)] If $\End(A)$ and $\End(B)$ are first countable and satisfy
            $\calP\in\Ptop$, then so does $\End(C)/I_E$.
            \item[(ii)] If $\End(A)$ and $\End(B)$ are quasi-$\pi$-regular,
            then so does $\End(C)/I_E$.
            \item[(iii)] If $\End(A)$ and $\End(B)$ are quasi-$\pi$-regular and
            semiperfect, then $\End(C)$ is semiperfect.
        \end{enumerate}
    \end{thm}

    In light of the previous results, one might wonder under what
    conditions all right f.p.\ modules over an LT ring are Hausdorff. This is treated
    in the next section and holds, in particular, for right noetherian pro-semiprimary rings and rank-$1$
    complete valuation rings.
    \smallskip

    We finish this section by noting that we can take a different approach
    for  complete Hausdorff modules. For the following discussion, a right $R$-module module $M$ is
    called \emph{complete} if the natural map $M\to \invlim \{M/MJ\}_{J\in\calI_R}$ is an isomorphism.\footnote{
        Completeness can also be defined for non-Hausdorff topological abelian groups; see \cite{Wa93}.
    }

    \begin{prp}
        (i) Let $R$ be a complete first countable Hausdorff LT ring. Then any \emph{Hausdorff}
        f.g.\ right $R$-module is complete.

        (ii) Let $\calP\in\Pdisc$ and let $R$ be an LT ring such that $R/I$ has
        $\calP$ for all $I\in \calI_R$. Let $M$ be a complete right $R$-module such
        that $M/JM$ is f.p.\ as a right $R/J$-module for all $J\in \calI_R$ (e.g.\ if $M$ is f.p.,
        or if $M$ is f.g.\ and $R$ is strictly pro-right-artinian). Then $\End(M)$ is pro-$\calP$ w.r.t.\ $\tau_1^M$.
        If moreover $R$ is semiperfect and $M$ is f.g., then $\End(M)$ is semiperfect.
    \end{prp}

    \begin{proof}
        (i) This is a well-known argument: Let $\calB$ be a countable local basis of $R$ consisting of ideals.
        Without loss of generality, we may assume $\calB=\{J_n\}_{n=1}^\infty$ with $J_1\supseteq J_2\supseteq\dots$.
        Let $M$ be a f.g.\ Hausdorff $R$ module and let
        $\{x_1,\dots,x_n\}$ be a set of generators of $M$. Since $M$ is Hausdorff, it is enough to
        show that any sum
        $\sum_{i=1}^\infty m_i$ with $m_i\in MJ_i$ converges in $M$. Indeed, write $m_i=\sum_{j=1}^n{x_jr_{ij}}$
        with $r_{i1},\dots,r_{in}\in J_i$. Then $\sum_{i=1}^\infty{r_{ij}}$ converges in $R$ for all $j$, hence
        $\sum_{i=1}^\infty m_i$ converges to $\sum_{j=1}^n x_jr_j$ where $r_j=\sum_{i=1}^\infty{r_{ij}}$.

        (ii) Throughout, $J$ denotes
        an open ideal of $R$.
        We first note that if $M$ is f.p.\ then there is an exact sequence $R^n\to R^m\to M\to 0$ for some
        $n,m\in\N$. Tensoring it with $R/J$ we get $(R/J)^n\to (R/J)^m\to M/MJ\to 0$, implying $M/J$ is a f.p.\
        $R/J$-module. Next, if $M$ is f.g.\ and $R$ is strictly pro-right-artinian, then $M/MJ$ is a f.g.\
        module over $R/J$ which is right artinian, hence $M/J$ is f.p.\ over $R/J$.

        Now, since $M/MJ$ is f.p.\ over $R/J$, $\End(M/MJ)$ satisfies $\calP$ by Theorem \ref{RING:TH:application-to-fp-modules}.
        There is a natural map $\End(M)\to \End(M/MJ)$ whose kernel is $\Hom(M,MJ)$.
        Assign $\End(M)$ the standard topology.
        Then since $M$ is complete, $\Hom(M,M)\cong \invlim \{\End(M/MJ)\}_{J\in \calI}$ as topological rings and
        therefore, $\End(M)$ is pro-$\calP$.

        Finally, assume $M$ is f.g.\ and $R$ is semiperfect. Then by Proposition \ref{RING:PR:semiperfect-equiv-conds}, $M$ admits a projective
        cover $P$ which is easily seen to be finitely generated.
        Assume $M=M_1\oplus \dots \oplus M_t$. Then each $M_i$ is f.g.\ and thus has a projective cover
        $P_i$. Necessarily $P\cong P_1\oplus\dots\oplus P_t$. By Proposition \ref{RING:PR:morita-preserved-properties},
        $\End(P_R)$ is semiperfect, hence
        there is a finite upper bound on the cardinality of sets of orthogonal idempotents in it.
        This means $t$ is bounded and hence, $\End(M)$ cannot contain
        an infinite set of orthogonal idempotents. By Lemma \ref{RING:LM:main-lemma-II-quasi-pi-reg}(ii),
        this implies $\End(M)$ is semiperfect.
    \end{proof}

\section{LT Rings With Hausdorff Finitely Presented Modules}
\label{section:rings-with-Hausodroff-mods}

    In this section we present sufficient conditions on an LT ring guaranteeing all right f.p.\
    modules  are Hausdorff (w.r.t.\ the standard topology). The discussion leads to an interesting
    consequence about noetherian pro-semiprimary rings.

    \smallskip

    We begin by noting a famous result
    that solves the problem for many noetherian rings with the
    Jacobson topology. For proof and details, see \cite[Th.\ 3.5.28]{Ro88}.

    \begin{thm}{\bf (Jategaonkar-Schelter-Cauchon).} \label{RING:TH:Jat-Sch-Cau}
        Let $R$ be an almost fully bounded noetherian ring whose primitive images are artinian
        (e.g.\ a noetheorian PI ring). Assign $R$ the $\Jac(R)$-adic topology
        or any stronger linear ring topology. Then any f.g.\ right $R$-module is Hausdorff.
    \end{thm}

    \begin{example} \label{RING:EX:non-hausdorff-I}
        The assumption that \emph{all} powers of $\Jac(R)$ are open in the last theorem cannot be dropped:
        Let $R$ be a Dedekind domain with exactly two prime ideals $P$ and $Q$.
        Then $R$ is noetherian,
        almost fully bounded and any primite image of $R$ is artinian.
        Let $n\in\N\cup\{0\}$ and let $\calB=\{P^mQ^n\where m\in\N\}$. Assign $R$ the unique topology with
        local basis $\calB$. Clearly $\Jac(R)^k=P^kQ^k$ is
        open for all $1\leq k\leq n$. However, $\overline{\Jac(R)^{n+1}}=\overline{P^{n+1}Q^{n+1}}=\bigcap_{m=1}^\infty(P^{n+1}Q^{n+1}+P^mQ^n)=
        \bigcap_{m=1}^\infty(P^{\min\{n+1,m\}}Q^{n})=P^{n+1}Q^n$, so $\Jac(R)^{n+1}$ is not closed.
        In particular, by ($*$) below, $R/\Jac(R)^{n+1}$ is a f.g.\ non-Hausdorff $R$-module.
    \end{example}

    When considering quasi-$\pi$-regular rings, there is actually no point in taking a topology stronger than the Jacobson
    topology in Theorem \ref{RING:TH:Jat-Sch-Cau}, because for right noetherian rings the latter is the largest topology making the ring quasi-$\pi$-regular.

    \begin{prp}\label{RING:PR:Jacobson-topology-is-the-biggest}
        Let $R$ be an LT semilocal ring and let $\tau$ be the topology on $R$.
        Assume $R$ is quasi-$\pi$-regular w.r.t.\ $\tau$ and $R/I$
        is semiprimary for all $I\in\calI_R$ (e.g.\ if $R$ is right noetherian
        or pro-semiprimary w.r.t.\ $\tau$). Then $R$ is quasi-$\pi$-regular w.r.t.\
        the Jacobson topology and the latter contains $\tau$.
    \end{prp}

    \begin{proof}
        We first note that if $R$ is right noetherian, then for all $I\in \calI_R$, $R/I$ is right noetherian and $\pi$-regular,
        hence by Remark \ref{RING:RM:hirarchy-of-properties}, $R/I$ is semiprimary. If $R$ is pro-semiprimary,
        then $R/I$ is semiprimary for all $I\in\calI_R$ by Lemma \ref{RING:LM:pro-P-equiv-cond}.

        Let $\tau_{\Jac}$ denote the Jacobson topology and let $a\in R$. Then $a$ has an associated idempotent $e$
        w.r.t.\ $\tau$. Let $f=1-e$ and observe that $\Jac(R)$ is open by Remark \ref{RING:RM:Jacobson-radical-might-be-open}(ii).
        Then there is $n\in\N$ such that $(faf)^n\in\Jac(R)$ and it follows
        that $(faf)^n\xrightarrow{n\to\infty}0$ w.r.t.\ $\tau_{\Jac}$. Therefore, $e$
        is the associated idempotent of $a$ w.r.t.\ $\tau_{\Jac}$, hence
        $R$ is quasi-$\pi$-regular provided we can verify $\tau_{\Jac}$ is Hausdorff. This holds since
        $\tau_{\Jac}\supseteq \tau$, by the proof of Corollary \ref{RING:CR:pro-semiprim-is-satisfies-jacobson-conj}
        (which still works under our weaker assumptions).
    \end{proof}

    Stronger linear topologies are ``better'' since they have more Hausdorff modules.
    Note that the topology of an arbitrary qausi-$\pi_\infty$-regular
    ring can be stronger than the Jacobson topology. For example, take any non-semiprimary perfect ring
    $R$ with $\bigcap_{n\in\N}\Jac(R)=\{0\}$ (e.g.\ $R=\Q[x_1,x_2,x_3,\dots\where x_m^2=x_nx_m=0~\forall n>2m]$)
    and give it the discrete topology.

    \smallskip

    The next result will rely on following observation:
    \begin{enumerate}
        \item[($*$)] Let $R$ be an LT ring. If $M$ is a right $R$-module and $N$ is a submodule, then $\overline{N/N}=\overline{N}/N$.
        In particular, $M/N$ is Hausdorff if and only if $N$ is closed.
    \end{enumerate}
    Indeed, $\overline{N/N}=\bigcap_{J\in\calI_R}(M/N)J=\bigcap_{J\in\calI_R}(MJ+M)/N=(\bigcap_{J\in\calI_R}(MJ+M))/N=\overline{N}/N$.
    We will also need the following theorem. For proof, see \cite[\S7.4]{BoTheoryOfSets}.

    \begin{thm} \label{RING:TH:non-empty-inverse-limit}
        Let $\{X_i,f_{ij}\}$ be an $I$-indexed inverse system of non-empty
        sets. Assume that for each $i\in I$ we are given
        a family of subsets
        $T_i\subseteq P(X_i)$ such that for all $i\leq j$ in $I$ we have:
        \begin{enumerate}
            \item[(a)] $X_i\in T_i$ and $T_i$ is closed under (arbitrary large) intersection.
            \item[(b)] \emph{Finite Intersection Property:} If $L\subseteq T_i$ is such that the intersection of finitely many
            of the elements of $L$ is non-empty, then $\bigcap_{A\in L}A\neq \phi$.
            \item[(c)] For all $A\in T_j$, $f_{ij}(A)\in T_i$.
            \item[(d)] For all $x\in X_i$, $f_{ij}^{-1}(x)\in T_j$.
        \end{enumerate}
        Then $\invlim \{X_i\}_{i\in I}$ is non-empty.\footnote{
            This can be compared to the following topological fact: An inverse limit of an inverse system of non-empty Hausdorff compact
            topological spaces is non-empty and compact.
        }
    \end{thm}

    \begin{lem} \label{RING:LM:coset-intersection}
        Let $R$ be a ring and let $M$ be a right $R$-module. Let
        $\{M_i\}_{i\in I}$ be a family of submodules of $M$ and let $\{x_i\}_{i\in I}$ be elements of $M$.
        Then $\bigcap_{i\in I}(x_i+M_i)$ is either empty or a coset of $\bigcap_{i\in I}M_i$.
    \end{lem}

    \begin{proof}
        This is straightforward.
    \end{proof}

    \begin{thm} \label{RING:TH:closed-submodule-thm}
        Let $R$ be strictly pro-right-artinian. Then any f.g.\ submodule of a Hausdorff right $R$-module is closed.
    \end{thm}

    \begin{proof}
        Let $\calB$ be a local basis of ideals such that $R/J$ is right artinian for all $J\in\calB$.
        Assume $M$ is a Hausdorff right $R$-module, let $m_1,\dots, m_k\in M$
        and $N=\sum_{i=1}^km_iR$.
        We will show that $m\in\overline{N}$ implies $m\in N$.

        Let $m\in \overline{N}$. For every $J\in\calB$ define
        \[X_J=\left\{(a_1,\dots,a_k)\in (R/J)^k\suchthat \sum_i(m_i+MJ)a_i=m+MJ\right\}~.\]
        Observe that $m\in\overline{N}=\bigcap_{J\in\calB} (N+MJ)$, hence for all $J\in \calB$ there
        are $b_1,\dots,b_k\in R$ and $z\in MJ$ such that $\sum m_ib_i=m+z$, implying $X_J\neq\phi$.
        For all $J\subseteq I$ in $\calB$, let $f_{I\!J}$ denote the map from $(R/J)^k$ to $(R/I)^k$
        given by sending $(b_1+J,\dots,b_k+J)$ to $(b_1+I,\dots,b_k+I)$. Then $f_{I\!J}(X_J)\subseteq X_I$.
        It easy to check that $\{X_I,f_{I\!J}|_{X_J}\}$ is an inverse system of sets.

        For all $J\in\calB$, define $T_J$ to be the set consisting of the empty set together
        with all cosets of (right) $R$-submodules of $(R/J)^k$ contained in $X_J$. We claim that
        conditions (a)-(d) of Theorem \ref{RING:TH:non-empty-inverse-limit} hold. Indeed, $X_J$ is
        easily seen to be a coset of a submodule of $(R/J)^k$, thus  $X_J\in T_J$. In addition, by Lemma \ref{RING:LM:coset-intersection},
        $T_J$ is closed under intersection, so (a) holds. Since $R/J$ is right artinian, so is  $(R/J)^k$
        (as a right $R$-module).
        Lemma \ref{RING:LM:coset-intersection} then implies that cosets of submodules of $(R/J)^k$ satisfy the DCC, hence (b) holds. Conditions (c) and (d)
        are straightforward. Therefore, we may apply Theorem \ref{RING:TH:non-empty-inverse-limit} to deduce that
        $\invlim X_J$ is non-empty.

        Let $x\in \invlim\{X_J\}_{J\in\calB}$. Then $x$ consists of tuples $\{(a_1^{(J)},\dots,a_k^{(J)})\in (R/J)^k\}_{J\in \calB}$ that
        are compatible with the maps $\{f_{I\!J}\}$. As $R$ is complete, there are $b_1,\dots,b_k\in R$
        such that $a_i^{(J)}=b_i+J$ for all $1\leq i\leq k$ and $J\in\calB$. It follows that $m-\sum_{i}m_ib_i\in \bigcap_{J\in\calB}MJ$.
        As $M$ is Hausdorff, the right-hand side is $\{0\}$, so $m=\sum_{i}m_ib_i\in N$.
    \end{proof}

    \begin{remark}\label{RING:RM:f-cogenerated-rings}
        Theorem \ref{RING:TH:closed-submodule-thm} and its consequences actually hold for
        the larger class of
        strictly pro-\emph{right-finitely-cogenerated} rings.
        A module $M$ over a ring $R$ is called \emph{finitely cogenerated}\footnote{
            Other names used in the literature are ``co-finitely generated'', ``finitely embedded'' or ``essentially artinian''.
        } (abbrev.: f.cog.) if its submodules satisfy the Finite Intersection Property (condition (b) in Theorem
        \ref{RING:TH:non-empty-inverse-limit}). This is equivalent to $\soc(M)$ being f.g.\ and essential in $M$ (see
        \cite[Pr.\ 19.1]{La99}). A ring $R$ called \emph{right f.cog.}\ if $R_R$ is finitely cogenerated.
        (For example, any right pseudo-Frobeniuos ring is right f.cog.) Among the examples of
        strictly pro-\emph{right-finitely-cogenerated} rings are complete rank-$1$ valuation rings. Indeed, if $\nu:R\to \R$ is an (additive) valuation,
        and $R$ is complete w.r.t.\ $\nu$, then $R=\invlim \{R/\{x\in R\where\nu(x)>n\}\}_{n\in\N}$.
        For a detailed discussion about f.cog.\ modules and rings, see \cite{Va68}
        and \cite[\S19]{La99}.
    \end{remark}

    Notice that a complete semilocal ring is strictly pro-right-artinian if and only
    if its Jacobson radical is f.g.\ as a right module. The latter condition is commonly
    used when studying complete semilocal rings
    (e.g.\ \cite{Ro87}). In particular, Hinohara proved
    Theorem \ref{RING:TH:closed-submodule-thm} for complete semilocal rings satisfying it (\cite[Lm.\ 3]{Hi60}).
    (Other authors usually assume the ring is right noetherian.)
    By ($*$) we now get:

    \begin{cor} \label{RING:CR:fp-modules-over-pro-right-artinian}
        Let $R$ be a strictly pro-right-artinian ring. Then any f.p.\ right $R$-module is Hausdorff.
    \end{cor}

    We can now prove that under mild assumptions, strictly pro-right-artinian rings are complete semilocal.

    \begin{cor}
        Let $R$ be a strictly pro-right-artinian ring.
        If $J\subseteq \Jac(R)$ is an ideal that is
        f.g.\ as a right ideal, then $R$ is complete in the $J$-adic topology
        (i.e.\ $R\cong\invlim \{R/J^n\}_{n\in\N}$). If moreover $R/J$ is right artinian, then the topology
        on $R$ is the Jacobson topology!
        In particular, if $\Jac(R)$ is f.g.\ as a right ideal, then the topology
        on $R$ is the Jacobson topology and $R$ is complete semilocal.
    \end{cor}

    \begin{proof}
        Let $\calB$ be a local basis of ideals of $R$ such that $R/I$ is right artinian for all $I\in\calB$.
        We identify $R$ with its natural copy in $\prod_{I\in\calB}R/I$.
        Since $J$ is f.g.\ as a right ideal, then so are its powers. Therefore,
        by Theorem \ref{RING:TH:closed-submodule-thm}, $J^n$ is closed for all $n\in\N$.

        Let $\vphi$ denote the standard map from $R$ to $\invlim \{R/J^n\}_{n\in\N}$.
        Define a map $\psi:\invlim \{R/J^n\}_{n\in\N}\to R$ as follows: Let $r\in \invlim \{R/J^n\}_{n\in\N}$
        and let $r_n$ denote the image of $r$ in $R/J^n$.
        By Corollary \ref{RING:CR:pro-semiprim-is-satisfies-jacobson-conj}, for all $I\in \calB$,
        there is $n\in\N$ (depending on $I$) such that $J^n\subseteq I$.
        Let $r_I$ denote the image of $r_n$ in $R/I$. It is easy to check
        that $r_I$ is independent of $n$ and that $\what{r}:=(r_I)_{I\in\calB}\in R$. Define
        $\psi(r)=\what{r}$.

        It is straightforward to check that $\psi\circ \vphi=\id$. Therefore, we are done if we show that $\psi$ is injective.
        Let $y\in\ker \psi$ and let $y_n+J^n$ be the image of $y$ in $R/J^n$. Then for all $I\in\calB$,
        $J^n\subseteq I$ implies $y_n\in I$. This means
        $y_n\in\bigcap_{J^n\subseteq I\in\calB}I=\overline{J^n}=J^n$, so $y_n+J^n=0+J^n$ for all $n\in\N$,
        hence $y=0$.

        Now assume $R/J$ is right artinian. Then $\Jac(R)^k\subseteq J\subseteq\Jac(R)$ for some $k\in\N$,
        hence the Jacobson topology and the $J$-adic topology coincide.
        By Proposition \ref{RING:PR:Jacobson-topology-is-the-biggest}, the topology
        on $R$ is contained in the Jacobson topology, so we only need to show
        the converse. Let $n\in\N$. It is enough to show that $J^n$ is open.
        Indeed, since $J_R$ is f.g., then so is $(J^i/J^{i+1})_R$ ($i\geq 0$). As $(R/J)_R$ has finite length,
        $(J^i/J^{i+1})_R$ has finite length. Thus, $(R/J^n)_R$ have finite length as well.
        Since $J^n$ is closed, $J^n$ is an intersection of open ideals. As $(R/J^n)_R$ is of finite length,
        $J^n$ is the intersection of finitely many of those ideals, hence open.
    \end{proof}

    \begin{cor} \label{RING:CR:noetherian-pro-semiprimary-rings}
        Let $R$ be a right noetherian pro-$\pi$-regular ring. Then
        the topology on $R$ is the Jacobson topology, $R$ is strictly pro-right-artinian w.r.t.\ it
        and any right ideal of $R$ is closed. In particular, $R$ is semilocal complete.
    \end{cor}

    \begin{proof}
        By Lemma \ref{RING:LM:pro-P-equiv-cond}, $R/I$ is $\pi$-regular for all $I\in\calI_R$, hence
        Remark \ref{RING:RM:hirarchy-of-properties} implies $R/I$ is right artinian for all $I\in\calI_R$
        (since $R/I$ is right noetherian). Therefore, $R$ is pro-right-artinian,
        with $\Jac(R)_R$ finitely generated. Now apply the previous corollary.
\rem{
        so by Proposition \ref{RING:PR:Jacobson-topology-is-the-biggest}, $\tau\subseteq\tau_{\Jac}$.
        Since $R$ is right noetherian, any right ideal
        is f.g., hence closed by Theorem \ref{RING:TH:closed-submodule-thm}. In
        particular, $\Jac(R)^n$ is closed for all $n\in\N$ and is thus an intersection of open ideals.
        But $R/\Jac(R)^n$ is right artinian (since it is semiprimary and right noetherian),
        so $\Jac(R)^n$ is the intersection of finitely many open ideals. Therefore $\Jac(R)^n$ is open
        for all $n$, implying $\tau_{\Jac}\subseteq \tau$.}
    \end{proof}

    The next example demonstrates that Theorem \ref{RING:TH:closed-submodule-thm} fails for
    pro-artinian rings (and in particular for pro-semiprimary rings). It also
    implies that there are pro-artinian rings that are not strictly pro-right-artinian.

    \begin{example} \label{RING:EX:JacRsquared-can-be-non-open}
        Let $S=\Q(x)[t\where t^3=0]$. For all $n\in\N$ define $R_n=\Q(x^{2^n})+\Q(x)t+\Q(x)t^2\subseteq S$ and
        $I_n=\Q(x^{2^n})t^2\subseteq S$. Then $R_n$ is an artinian ring and $I_n\idealof R_n$. For $n\leq m$
        define
        a map $f_{nm}:R_m/I_m\to R_{n}/I_{n}$ by $f_{nm}(x+I_{m})=x+I_{n}$.
        Then $\{R_n/I_n,f_{nm}\}$ is an inverse system of artinian rings. Let $R=\invlim \{R_n/I_n\}_{n\in\N}$.
        Then $R$ can be identified with $\Q+\Q(x)t+Vt^2$ where $V$ is the $\Q$-vector space $\invlim \{\Q(x)/\Q(x^{2^n})\}_{n\in\N}$
        ($R$ does not embed in $S$).
        Observe that $\Q(x)$ is dense in $V$, but $\Q(x)\neq V$ since $V$ is not countable
        (it contains a copy of all power series $\sum a_nx^{2^n}\in \Q\dbrac{x}$). Therefore,
        the ideal $tR=\Q t+\Q(x)t^2$ is not closed in $R$ and by ($*$), $R/tR$ is a non-Hausdorff
        f.p.\ module.\rem{\footnote{
            The author is not sure whether $V$ can be naturally identified with $\Q\dpar{x}$. Even if the answer is yes,
            then the topology on $V$ is not likely to be the $x$-adic topology.
        }} We also note that $\Jac(R)^2=\Q(x)t^2$ is not closed (but $\Jac(R)$ must be closed by
        Proposition \ref{RING:LM:closed-jacobson-radical}).
    \end{example}

    We conclude by specializing the results of the previous section to first countable strictly pro-right-artinian
    rings. (We are guaranteed that all f.p.\ modules are Hausdorff in this case). By Corollary \ref{RING:CR:noetherian-pro-semiprimary-rings},
    this family include all right
    noetherian pro-semiprimary rings. More general statements can be obtained by applying Remark \ref{RING:RM:f-cogenerated-rings}.

    \begin{cor}
        (i) Let $R$ be a first countable pro-right-artinian ring and let $M$
        be a f.p.\ right $R$-module. Then $\End(M_R)$ is
        pro-semiprimary and first countable (w.r.t.\ $\tau_2^M$). If $R$ is semiperfect (e.g.\ if $R$ is right noetherian),
        then $\End(M_R)$ is semiperfect.

        (ii) Let $S$ be a commutative first countable pro-right-artinian ring and let $R$ be an $S$-algebra s.t.\ $R$ is f.p.\
        as an $S$-module. Then $R$ is pro-semiprimary (w.r.t.\ some topology). If $S$ is semiperfect (e.g.\ if $S$ is right noetherian),
        then $R$ is semiperfect.
    \end{cor}

\section{Further Remarks}
\label{section:further-remarks}

    It is likely that the theory of semi-invariant subrings developed in section \ref{section:top-rings}
    can be extended to \emph{right} linearly topologized rings, i.e.\ topological rings having a local
    basis consisting of \emph{right} ideals. This actually has the following remarkable implication
    (compare with Corollary \ref{RING:CR:application-to-fp-modules} and Corollary \ref{RING:CR:fp-algebras-full-cor-topological-version}):

    \begin{cnj}
        Let $S\in\LTR$ be a semiperfect quasi-$\pi_\infty$-regular ring and let $\vphi:S\to R$ be a ring homomorphism. Assume that:
        \begin{enumerate}
        \item[(a)] When considered as a right $S$-module via $\vphi$, $R$ is f.p.\ and Hausdorff.
        \item[(b)] For all $r\in R$ and $I\in\calI_S$, there is $J\in\calI_S$ such that $R\vphi(J)r\subseteq R\vphi(I)$.\footnote{
            This equivalent to saying that the topology on $R$ spanned by
            cosets of the left ideals $\{R\vphi(I)\where  I\in\calI_S\}$ is a ring topology;
            see \cite[\S3]{Wa93}.
        }
        \end{enumerate}
        Then $R$ is semiperfect and quasi-$\pi_\infty$-regular (w.r.t.\ some topology).
    \end{cnj}

    The proof should be as follows: For any right $S$-module $M$, let $W$ denote the ring of
    continuous $\Z$-homomorphisms from $M$ to itself. Then $W$ can be made into a \emph{right LT} ring by
    taking $\{\Ball(J)\where J\in \calI_S\}$ as a local basis where
    $\Ball(J)=\{f\in W\suchthat \im f\subseteq MJ\}$ (this is the topology
    of uniform convergence).\footnote{
        Caution: Not any filter base of right ideals gives rise to a ring topology.
        By \cite[\S3]{Wa93}, we need to check that for all $f\in W$
        and $I\in\calI_S$ there is $J\in\calI_S$ such that $f\Ball(J)\subseteq\Ball(I)$. Indeed, we
        can take any $J$ with $f(MJ)\subseteq MI$ and such $J$ exists since $f$ is continuous.
    } Clearly $W$ contains $\End(M_S)$ as a topological ring (endow $\End(M_S)$
    with $\tau_1^M$). Now take $M=R$ (where $R$
    is viewed as a right $S$-module via $\vphi$). Then condition (a) implies $\End(R_S)$
    is semiperfect and quasi-$\pi_\infty$-regular w.r.t.\ $\tau_1^R$ (Theorem \ref{RING:TH:application-to-fp-modules-top-case}).
    Condition (b) implies that for all $r\in R$, the map $\what{r}:x\mapsto xr$ from $R$ to itself is continuous and
    hence lie in $W$. Since we assume the results of section \ref{section:top-rings}
    extend to right LT rings, $R\cong \End(R_R)=\Cent_{\End(R_S)}(\{\what{r}\where r\in R\})$ is a T-semi-invariant
    subring of $\End(R_S)$, so $R$ is semiperfect and quasi-$\pi_\infty$-regular.

    \smallskip

    Examples of rings satisfying conditions (a) and (b) can be produced by taking $R$ to be:
    (1) a twisted group algebra $S^\alpha G$
    where $G$ is a \emph{finite} group and $\alpha:G\to\cAut(G)$ is a group homomorphism
    or (2) a ``crossed product'', i.e.\ $R=\mathrm{CrossProd}(S,\psi,G)$ where $S$
    is commutative, $G$ is finite and act on $S$ via continuous automorphisms and $\psi\in H^2(G,\units{S})$.
    (Further examples
    can be produced by taking quotients). However, we can show directly that the conjecture holds in these special
    cases. Indeed, that $G$ is finite implies $\calB=\{\bigcap_{g\in G}g(I)\where I\in\calI_R\}$ is
    a local basis of $S$ and we have
    $RJ\subseteq JR$ for all $J\in\calB$. For any right $R$-module $M$ let
    $W'=\{f\in W\suchthat f(MJ)\subseteq MJ~\forall J\in\calB\}$ (with $W$ as in the previous paragraph).
    Then $W'$ is a
    \emph{linearly topologized} ring w.r.t.\
    the topology induced from $W$ (as seen by taking the local basis $\{W'\cap\Ball(J)\where J\in\calB\}$).
    In addition, when $M=R_S$,
    $\what{r}$ of the previous paragraph lies in $W'$ (since $RJ\subseteq JR$ for all $J\in\calB$). Therefore, repeating
    the argument of the last paragraph with $W'$
    instead of $W$, we get that $R$ is semiperfect and quasi-$\pi_\infty$-regular.

    \medskip

    The author could not find nor contradict the existence of the following:
    \begin{enumerate}
        \item[(1)] A pro-semiprimary ring that is not complete semilocal (i.e.\ complete w.r.t.\ its
        Jacobson topology).
        \item[(2)] A complete semilocal ring, endowed with the Jacobson topology, with a non-Hausdorff f.p.\ module.
    \end{enumerate}

\section{Appendix: When Do $\tau_1^M,\tau_2^M,\dots$ Coincide?}

    This appendix is dedicated to the question of when the topology obtained from a resolution
    is the standard topology.
    For that purpose we briefly recall the Artin-Rees property for ideals. For details and proofs, see
    \cite[\S3.5D]{Ro88}.

    Let $R$ be a \emph{right noetherian} ring.
    An ideal $I\idealof R$ is said to satisfy the \emph{Artin-Rees property} (abbreviated: AR-property)
    if for any right ideal $A \leq R$ there is $n\in\N$ such that $I^n\cap A\subseteq AI$.
    This is well known to imply that for any f.g.\ right $R$-module $M$ and a submodule
    $N$, there is $n\in\N$ such that $MI^n\cap N\subseteq NI$.
    For example, by \cite[p.\ 462, Ex.\ 19]{Ro88},
    every \emph{polycentral} ideal (e.g.\ an ideal generated by central elements) satisfies the AR-property. In addition,
    if $R$ is \emph{almost bounded} (e.g.\ a PI ring), then all ideals of $R$ are satisfy the AR-property.

    Now let $R$ be any LT ring. A right $R$-module $M$ is said to satisfy the \emph{topological Artin-Rees property}
    (abbreviated: TAR-property) if for any submodule $N\subseteq M$ and any $I\in \calI_R$
    there
    is $J\in\calI_R$ such that $MJ\cap N\subseteq NI$. (Equivalently,
    the induced topology and the standard topology coincide for any submodule of $M$.)
    For example, if $R$ is right noetherian, $J\idealof R$
    and $R$ is given the $J$-adic topology, then all f.g.\ right $R$-modules satisfy  the TAR-property
    if and only if $J$ satisfies the AR-property.

    \begin{prp} \label{RING:PR:artin-rees-implications-I}
        Let $R$ be an LT ring, let $M$ be a right $R$-module
        and let $P: P_{n-1}\to\dots\to P_0\to P_{-1}=M\to 0$ be a projective
        resolution of $M$.
        Assume that
        that $P_0,\dots, P_{n-1}$
        have the TAR-property. Then $\tau_{P}$ is the natural topology on $\End(M)$.
    \end{prp}

    \begin{proof}
        Denote by $d_i$ the map $P_i\to P_{i-1}$ and let $B_i=\ker d_{i}$.
        We will prove that for all $I\in\calI_R$ there is $J\in\calI_R$ such that $\Hom(M,MJ)\subseteq \Ball(I,P)$.
        Given $I\in\calI_R$ we define a sequence of open ideals $I_{n-1},I_{n-2},\dots,I_{-1}$ as follows:
        Let $I_{n-1}=I$. Given $I_i$, take $I_{i-1}$ to be an open ideal such that $P_{i-1}I_{i-1}\cap B_{i-1}\subseteq B_{i-1}I_{i}$
        and $I_{i-1}\subseteq I_i$ (the existence of $I_{i-1}$ follows from the TAR-property).
        We claim that $\Hom(M,MI_{-1})\subseteq \Ball(I,P)$. To see this, let
        $f_{-1}\in\Hom(M,MI_{-1})$ and assume
        we have constructed maps $f_i\in\Hom(P_i,P_iI_i)$ for all $-1\leq i<k$ such
        that $d_if_i=f_{i-1}d_i$. Then it is enough to show there is $f_k\in\Hom(P_k,P_kI_k)$
        such that $d_kf_k=f_{k-1}d_k$. The argument to follow is illustrated in the following diagram:
        \[\xymatrix{
        P_k \ar@{->>}[r]^{d_{k}} \ar@{.>}[dd]^{f_k} & B_{k-1} \ar@{^{(}->}[r] \ar[d]^{f_{k-1}}             & P_{k-1} \ar[d]^{f_{k-1}} \ar[r]^{d_{k-1}} & \dots\\
                                                    & P_{k-1}I_{k-1}\cap B_{k-1} \ar@{^{(}->}[d]\ar@{^{(}->}[r] & P_{k-1}I_{k-1} \ar[r]^-{d_{k-1}} & \dots\\
        P_kI_k \ar@{->>}[r]^{d_{k}}                 & B_{k-1}I_k                              &                                 &
        }\]
        That $d_{k-1}f_{k-1}=f_{k-2}d_{k-1}$
        implies $f_{k-1}(B_{k-1})\subseteq B_{k-1}$. As $\im f_{k-1}\subseteq P_{k-1}I_{k-1}$, we
        get that $f_{k-1}(B_{k-1})\subseteq P_{k-1}I_{k-1}\cap B_{k-1}\subseteq B_{k-1}I_k$ (by
        the definition of $I_{k-1}$). Since the map $P_kI_k\to B_{k-1}I_{k}$ is onto (because $\im(d_k)=B_{k-1}$), we
        can lift $f_{k-1}d_k:P_k\to B_{k-1}I_k$ to module homomorphism $f_k:P_k\to P_kI_k$, as required.
    \end{proof}

    \begin{remark}
        The proof still works if we replace the assumption that $P_{n-1}$ is projective
        with $d_{n-1}$ is injective.
    \end{remark}

    \begin{cor} \label{RING:CR:artin-rees-implications-II}
        Let $R$ be an LT right noetherian ring admitting local basis of ideals $\calB$ such
        that: (1) all ideals in $\calB$ have the AR-property (e.g.\ if all ideals in $\calB$
        are generated by central elements or if $R$ is PI) and (2)
        all powers of ideals in $\calB$ are open.
        Then $\tau_1^M=\tau_2^M=\dots$ for any f.g.\ right $R$-module $M$.
    \end{cor}

    \begin{proof}
        Let $n\in\N$. Since $R$ is right noetherian, any f.g.\ $R$-module admits
        a resolution of length $n$ consisting of f.g.\ projective modules.
        The assumptions (1) and (2) are easily
        seen to imply that any f.g.\ $R$-module satisfies the TAR-property.
        Therefore, by Proposition \ref{RING:PR:artin-rees-implications-I}, $\tau_n^M$ is the
        natural topology on $\End(M)$.
    \end{proof}

    \begin{example}
        Condition (2) in Corollary \ref{RING:CR:artin-rees-implications-II} is essential
        even when all ideals of $R$ have the AR-property:
        Assign $\Z$ the unique topology with local basis
        $\calB=\{2\cdot 3^n\Z\where n\geq 0\}$ and let $M=\Z/4\times \Z/2$.
        Since $MI=2M$ for all $I\in\calB$,
        the standard topology on $\End(M)$ is obtained from
        the local basis $\{\Hom(M,2M)\}$. ($M$ is not Hausdorff).
        Let $I=2\Z\in\calB$ and consider the projective resolution
        \[P:~4\Z\times 2\Z\hookrightarrow \Z\times\Z\to M\to 0~.\]
        Define $f_{-1}:M\to MI=2M$ by $f(x+4\Z,y+2\Z)=(2y+4\Z,0)$.
        Then any lifting $f_0\in\End(\Z\times\Z)$  of $f_{-1}$ must satisfy $f_0(0,1)=(4x+2,2y)$
        for some $x,y\in\Z$. This means that any lifting $f_1\in\End(4\Z\times 2\Z)$ of $f_0$  (there
        is only one such lifting) satisfies $f_1(0,2)=(8x+4,4y)\notin 8\Z\times 4\Z=(4\Z\times 2\Z)I$.
        Therefore, $f_1\notin \Hom(4\Z\times 2\Z, (4\Z\times 2\Z)I)$, implying
        $f_{-1}\notin\Ball(P,I)$. But this means that $\Ball(P,I)\subsetneq\Hom(M,2M)$, hence
        $\tau_2^M\neq\tau_1^M$.
    \end{example}

\section*{Acknowledgements}

    The author thanks his supervisor Uzi Vishne for his help and guidance,
    to Louis H.\ Rowen for his useful advice, to Eli Matzri for his help in
    choosing the appropriate values in Examples \ref{RING:EX:inv-subring-semiperfect-case-I} and \ref{RING:EX:inv-subring-semiperfect-case-II}
    and to Michael Megrelishvili and Menny Shlonssberg for their help with anything that involves topology.
    Ofir Gorodetsky, Tomer Schlank and others have contributed to the formulation and proof of Theorem \ref{RING:TH:closed-submodule-thm}
    and the author is grateful for their help as well. Finally, the author also thanks the referee for his beneficial comments.

\bibliographystyle{plain}
\bibliography{MyBib}

\end{document}